\documentclass[12pt,reqno,a4paper]{amsart}%{article}%

\usepackage{amssymb}

\usepackage{graphicx}

\usepackage{color}

\usepackage{subfig}

\usepackage{pgf,tikz}

\usetikzlibrary{arrows}

\usepackage[draft]{changes}

\numberwithin{equation}{section}
{\theoremstyle{definition}
\newtheorem{defn}{Definition}[section]}
\newtheorem{theorem}{Theorem}[section]
\newtheorem{proposition}[theorem]{Proposition}
\newtheorem{conjecture}[theorem]{Conjecture}
\newtheorem{corollary}[theorem]{Corollary}
\newtheorem{lemma}[theorem]{Lemma}

{\theoremstyle{definition}%{plain}
{%\theorembodyfont{\normalfont\rmfamily}
\newtheorem{remark}[theorem]{Remark}

}}
\newcommand{\cal}{\mathcal}

\newcommand{\PP}{{\cal P}}

\newcommand{\SSS}{{\cal S}}

\newcommand{\Cc}{{\mathbb{C}}}

\newcommand{\Nn}{{\mathbb{N}}}

\newcommand{\Rr}{{\mathbb{R}}}

\newcommand{\Zz}{{\mathbb{Z}}}

\def\dist{\operatorname{dist}}

\def\Re{\operatorname{Re}}
\def\Im{\operatorname{Im}}
\newcommand{\diam}{\operatorname{diam}}

\newcommand{\de}{\delta}\newcommand{\De}{\Delta}
\def\la{\lambda}

\newcommand{\vla}{{\boldsymbol{\la}}}
\newcommand{\vw}{{\boldsymbol{w}}}

\newcommand{\om}{\omega}\newcommand{\Om}{\Omega}

\def\R{\mathbb R}

\newcommand{\comment}[1]{}

\graphicspath{ {figs/} }

\begin{document}

\title[Outer Billiards]
{Dissipative outer billiards: a case study}

\author[G.~Del~Magno]{Gianluigi Del Magno}
\author[J.~P.~Gaiv\~ao]{Jos\/e Pedro Gaiv\~ao}
\author[E.~Gutkin]{Eugene Gutkin}
\dedicatory{This paper is dedicated to the memory of the third author, who unexpectedly passed away while the paper was being completed.}
\thanks{The first two authors were partially supported by Fundac\~ao para a Ci\^encia e a Tecnologia through the Program POCI 2010 and the Project `Randomness in Deterministic Dynamical Systems and Applications' (PTDC- MAT-105448-2008).}

\address{CEMAPRE, ISEG, Universidade de Lisboa, Rua do Quelhas 6, 1200-781
Lisboa, Portugal} 
\email{delmagno@iseg.utl.pt}

\address{CEMAPRE, ISEG, Universidade de Lisboa, Rua do Quelhas 6, 1200-781 Lisboa, Portugal}
\email{jpgaivao@iseg.utl.pt}

\address{Department of Mathematics, Nicolaus Copernicus University, Chopina 12/18, Torun 87-100, Poland;
\ Institute of Mathematics of Polish Academy of Sciences,
Sniadeckich 8, Warsaw 00-956, Poland}
\email{gutkin@mat.umk.pl, gutkin@impan.pl}



\date{\today}

\begin{abstract}
We study dissipative polygonal outer billiards, i.e. outer billiards about convex polygons with a contractive reflection law. We prove that dissipative outer billiards about any triangle and the square are asymptotically periodic, i.e. they have finitely many global attracting periodic orbits. A complete description of the bifurcations of the periodic orbits as the contraction rates vary is given. For the square billiard, we also show that the asymptotic periodic behavior is robust under small perturbations of the vertices and the contraction rates. Finally, we describe some numerical experiments suggesting that dissipative outer billiards about regular polygon are generically asymptotically periodic.
\end{abstract}

\maketitle

%\tableofcontents

\section{Introduction and motivation}       
\label{intro}

The {\em outer billiard} about a convex compact region $\Om \subset \Cc $ is the map $T \colon \Cc\setminus\Om\to\Cc\setminus\Om$ introduced by M.~Day~\cite{MD}. For $z\in X:=\Cc\setminus\Om$, let $L_z \subset\Cc$ be the ray through $z$ intersecting tangentially $\Om$ on the right with respect to the obvious orientation. Suppose first that $\Om$ is strictly convex, and let $t\in\partial\Om$ be the intersection point of $\Om$ and $L_z$. Then $T(z)\in L_z$ is the point symmetric to $z$ with respect to $t$. The transformation $T \colon X\to X$ is an area preserving twist map~\cite{KH}. Neumann asked if there are points $z\in X$ whose $\omega$-limit sets contain subsets of $\partial\Om$ or infinity~\cite{Neumann}. Applying his famous KAM-type theorem, Moser showed that this is impossible if $\partial\Om$ is continuously differentiable $333$ times and is strictly convex~\cite{Moser0,Moser1}. In this case, in fact, $X$ contains invariant circles arbitrarily close to $\partial\Om$ and arbitrarily far away from it. It follows that the orbit of every $z\in X$ belongs to an annulus bounded by two invariant circles. Later on, R.~Douady showed that six continuous derivatives of $\partial\Om$ suffice for the existence of invariant circles~\cite{Douady}. However, the strict convexity of $\partial\Om$ is necessary~\cite{GK}.

The definition of the outer billiard map has to be modified if $ \Omega $ is not strictly convex. Suppose for example that $ \Omega $ is a convex $ k $-gon $ P $. For points $z\in X$ such that the ray $L_z$ contains a side of $P$, the image $T(z)$ is not uniquely defined. These points form the {\em singular set} $\, \SSS$ of $T$, which is a union of $k$ half-lines. The outer billiard map is not defined on $\SSS$, and has jump discontinuities across its lines. The complement $X\setminus\SSS$ is a disjoint union of $k$ open cones $ A_{1},\ldots,A_{k} $. The apex of $A_k$ is a corner point $w_k\in P$, and the restriction $T|_{A_k}$ is the euclidean reflection about $w_k$. Thus, the outer billiard about $P$ is a planar piecewise isometry on $k$ cones. The dynamics of some classes of piecewise isometries\footnote{Interval exchanges transformation are one-dimensional piecewise isometries. This very special class of piecewise isometries makes a research subject on its own right, and it would take us too far afield to comment on it.} have been studied in detail, mostly in two dimensions~\cite{AsGo,Bu,GH1,GH2,Ha}. 

Polygonal outer billiards make a very special and beautiful class of piecewise isometries. Moser\footnote{Moser presented this question as a caricature of the problem on the stability of motion in celestial mechanics~\cite{Moser2}.} pointed out that for these billiards, Neumann's question about the existence of orbits accumulating on the boundary of the table or at infinity was open~\cite{Moser2}. In view of Moser's remark, the question has attracted much of attention, and several partial results in both directions were obtained~\cite{GuSi,Ko,Sch1,ViSh}. However, the full scope of the Neumann-Moser question remains open. For various aspects of polygonal outer billiards, we refer the reader to~\cite{BeCa,Gu12,GuTa,Hu,Ta,Sch2}.

We will now embed the outer billiard about the convex $ k $-gon $P$ into a one-parameter family of dynamical systems. As before, let $ A_{i} $ be the open cones of $ \Cc $ whose apex is the corner $ w_{i} $ of $ P $. Let $ \la \in (0,+\infty) $. Define
\begin{equation}    
\label{eq:Tlambda}
T_{\la}(z)=-\la z + (1+\la)w_{i} \quad \text{for } z \in A_{i}.
\end{equation}
The map $ T_{\la} $ is a piecewise affine dilation on the cones $ A_{1},\ldots,A_{k} $. For $\la=1$, we recover the outer billiard map. For  $0<\la<1$ (resp. $1<\la$), the map is piecewise contracting (resp. piecewise expanding). If we denote by $ Q $ the polygon obtained by reflecting $ P $ about the origin, then it is easy to see that $ T^{-1}_{\la} $ is conjugated to the map $ T'_{1/\la} $ of the outer billiard about $ Q $ with contraction rate $ 1/\la $. In particular, the $ \omega $-limit sets of $ T_{\la} $ and $ T'_{1/\la} $ coincide. Therefore, as far as the asymptotic properties of $ T_{\la} $ for $ \la \neq 1 $ are concerned, we will only consider the case $ 0 < \la < 1 $. 

We will call the map $ T_{\la} $ with $ 0 < \la < 1 $ the {\em dissipative outer billiard}, and will call {\em the outer billiard} the map $ T_{\la} $ with $ \la = 1 $ (see Fig.~\ref{fig:poly}).
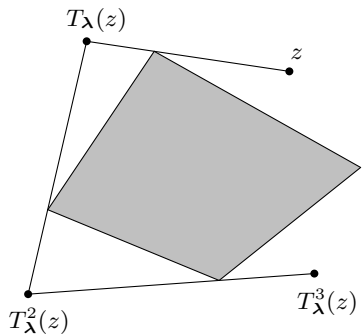
\begin{figure} 
%\begin{center} 
%\includegraphics[width=3.5in]{genpoly.eps} 
%\end{center} 
\definecolor{qqffff}{rgb}{0,1,1}
\begin{tikzpicture}[line cap=round,line join=round,>=triangle 45,x=.7cm,y=.7cm]
\clip(-1.12,-1.54) rectangle (6.54,5);
\fill[fill=black,fill opacity=0.25] (2,4) -- (0,1) -- (3.22,-0.34) -- (5.88,1.8) -- cycle;
\draw (2,4)-- (0,1);
\draw (0,1)-- (3.22,-0.34);
\draw (3.22,-0.34)-- (5.88,1.8);
\draw (5.88,1.8)-- (2,4);
\draw (4.54,3.62)-- (0.73,4.19);
\draw (0.73,4.19)-- (-0.37,-0.6);
\draw (-0.37,-0.6)-- (5.01,-0.21);
\begin{scriptsize}
\fill  (4.54,3.62) circle (1.5pt);
\draw (4.68,3.95) node {$z$};
\fill  (0.73,4.19) circle (1.5pt);
\draw (0.92,4.6) node {$T_{\vla}(z)$};
\fill  (-0.37,-0.6) circle (1.5pt);
\draw (-0.16,-1.1) node {$T^{2}_{\vla}(z)$};
\fill  (5.01,-0.21) circle (1.5pt);
\draw (5.26,-.75) node {$T^{3}_{\vla}(z)$};
\end{scriptsize}
\end{tikzpicture}
\caption{Example of a dissipative polygonal outer billiard.} 
\label{fig:poly} 
\end{figure} 

The dynamics of the outer billiard when $ \la = 1 $ about the equilateral triangle, the square and the hexagon is particularly simple: every orbit of these billiards is periodic~\cite{Moser2}. In this paper, we investigate in detail the dissipative outer billiards about the equilateral triangle and the square. We show that both billiards are {\em asymptotically periodic}, i.e., their $ \omega $-limit sets consist of periodic orbits, and every orbit converges to one of them. We then perform a bifurcation analysis, by computing the number and the period of these orbits as $ \la $ varies between 0 and 1. Using a result by Catsigeras and Budelli on general piecewise contractions~\cite{CB}, we extend the results obtained for the square to the more general situation when the contraction rate $ \la $ may depend on the vertices of $ P $, and to quadrilaterals sufficiently close to the square. By the affine equivariance of $ T_{\la} $, the results obtained in this paper are not limited to the equilateral triangle, the square and quadrilaterals close to the square, but extend also to the polygons obtained by applying to them affine transformations. 

Our motivation for studying dissipative polygonal outer billiards is threefold: 
\begin{enumerate}
\item similarly to the inner billiard \cite{Gu12}, little is known about periodic orbits of polygonal outer billiards \cite{Ta}. By a recent result of Culter, every outer polygonal billiard has a periodic orbit~\cite{TaCa}. A detailed knowledge of the periodic points of $T_{\la} $ for $\la<1 $ may help obtain information on periodic points when $ \la =1 $ by studying the limit $\la\to 1 $;
\item the dissipative polygonal inner billiard exhibits chaotic attractors with SRB measures~\cite{Arr_etal,DM_etal_1,DM_etal_2,Ma_etal}. It would be instructive to compare these phenomena with the dynamics exhibited by dissipative polygonal outer billiards;
\item by a result of Bruin and Deane, almost every piecewise contraction is
asymptotically periodic~\cite{Bruin}. It would be desirable to obtain a similar result for dissipative polygonal outer billiards.
\end{enumerate}

The paper is structured as follows. In Section~\ref{segment}, we discuss the dissipative outer billiard about a segment, which is in a sense a one-dimensional caricature of the dissipative polygonal outer billiard. Despite its somewhat artificial nature, this dynamical system shows in a nutshell some of the main features of our study. In Section~\ref{general}, we introduce some notation and collect a few general results about the periodic orbits and the $ \om $-limit set of dissipative outer billiards. The dissipative outer billiards about a triangle and about the square are studied in Sections~\ref{triangle} and \ref{square}, respectively. In Section~\ref{persistency}, we address the problem of the robustness of the asymptotic periodicity under perturbations of the polygon and the contraction rates. Finally, in Section~\ref{conclu}, we illustrate our numerical experiments and make a conjecture concerning dissipative outer billiards about general regular polygons.

While writing this paper, we learned from R.~Schwartz that Jeong obtained results similar to ours concerning the asymptotic periodicity of the square and the bifurcation analysis of its periodic orbits~\cite{J}. 

\section{One-dimensional billiards}  
\label{segment}

The outer billiards defined in the previous section are piecewise maps of $ \Cc $. In a similar fashion, one can define outer billiards on $ \Rr $. For these billiards, the choice of the region $ \Omega $ is limited, since $ \Omega $ can be only an interval $ I \subset \Rr $. In this subsection, we will see that the dissipative outer billiard on $\Rr$ yields a family of one-dimensional transformations, exhibiting typical features of the dissipative polygonal outer billiard.

We assume without loss of generality that $I=[-1,1]$, and set $X_+=(1,+\infty) $ and $ X_{-}=(\infty,-1)$. Then $X=X_+\cup X_-$ and $T_{\la} \colon X_{\pm}\to X_{\mp}$, where $ T_{\la}(x) = -\la x \mp (1+\la) $ on $ X_{\pm} $. In other words, every point of $ \Rr \setminus I $ is reflected about the further endpoint of $ I $. Since $T^{2}_{\la} (X_{\pm}) \subset X_{\pm}$, by symmetry, it suffices to analyze $T^{2}_{\la}|_{X_{+}} $. For $\la\in(0,1)$, the map $ T^{2}_{\la}|_{X_{+}} $ is a contraction. Its unique fixed point is $p_{\la}=(1+\la)/(1-\la)$. Furthermore, $\lim_{\la\to 0}p_{\la}=1 $ and $ \lim_{\la\to 1}p_{\la}=+\infty$. 

By the affine equivariance of our setting, the previous discussion yields the following.

\begin{proposition}  
\label{onedim_prop}
Let $I=[-a,a]\subset\R$ be any finite interval with $ a>0 $, and let $T_{\la} \colon \R\setminus I\to\R\setminus I$ be the dissipative outer billiard about $I$. For $0<\la<1$, the map $T_{\la}$ contracts $\R\setminus I$ with contraction coefficient $\la^2$ to the unique two-periodic orbit 
\[
\om_{\la}=\left\{a\frac{1+\la}{1-\la},-a\frac{1+\la}{1-\la}\right\}.
\]
Moreover, we have 
\[ 
\lim_{\la\to 0} \om_{\la}=\{a,-a\} \quad \text{and} \quad \lim_{\la\to 1} \om_{\la}=\{+\infty,-\infty\}.
\] 
\end{proposition}

In a nutshell, the one-dimensional dissipative outer billiard is a smooth family of linear contractions. By iteration, each $T_{\la}$ contracts the whole space to the unique periodic orbit. This orbit depends analytically on the dissipation parameter $\la$. As the parameter approaches the boundary $\{0,1\}$ of the range, the transformation $T_{\la}$ converges to the limit maps $T_0 $ and $ T_1$. The map $T_0$ is degenerate: $T_0(X)=\{-a,a\}$, whereas the map $T_1$ is the usual outer billiard about the interval $[-a,a]$. As the reader will see that the dissipative outer billiard about the equilateral triangle and the square retain these features, adding to them an infinite sequence of bifurcations.

\section{General results}   
\label{general}

We now extend the definition of the dissipative outer billiards, allowing it to have several contraction rates. Let $ P $ be a convex $ k $-gon with vertices $ \vw = (w_1,\dots,w_k) $, and let $ \vla = (\la_{1}, \ldots,\la_{k}) $ with $ 0 < \la_{i} < 1 $ for $ i=1,\ldots,k $. The sup norm of $ \Rr^{k} $ is denoted by $ \|\cdot\| $. We also denote by $B_{r}(z)\subset\Cc$ the closed ball of radius $r>0$ centered at $z \in \Cc $ with respect to the norm $ \|\cdot\| $.

As before, let $ X = \Cc \setminus P $, and let $ A_{i} $ be the cone with apex $ w_{i} $. The dissipative outer billiard map about $ P $ with contraction rates $ \vla $ is given by
\begin{equation}    
\label{eq:Tvlambda}
T_{\vla}(z)=-\la_{i} z + (1+\la_{i})w_{i} \quad \text{for } z \in A_{i}.
\end{equation}

As for the single contraction rate case, the singular set $ \SSS $ of $ T_{\vla} $ consists of the half-lines containing the sides of $ P $. For any integer $n\geq1$, let $\mathcal{S}_n$ denote the set of points $z\in X$ for which there exists an integer $0\leq i<n$ such that $T^i_\vla(z)$ belongs to $\mathcal{S}$ (so $ \SSS_{1} = \SSS $). We call $\mathcal{S}_n$ the \textit{singular set of order $n$}. Note that, $\mathcal{S}_n$ is a finite union of half-lines, and $\{\mathcal{S}_n\}$ is a filtration, i.e. 
$\mathcal{S}_n\subseteq\mathcal{S}_{n+1}$ for every $n\geq1$.

We say that a point $z \in X$ is {\em non-singular} if $z \notin \SSS_{n} $ for every $ n \ge 1 $. For such a point $ z $, we denote by $\om(z)$ the $\om$-limit set of the $ T_{\vla} $-orbit of $ z $. We also denote by $ \omega(T_{\vla}) $ the union of the $ \om $-limit sets of all non-singular points of $ T_{\vla} $. The following notion is central in this paper.

\begin{defn}
\label{de:ap}
We say that $T_\vla$ is \textit{asymptotically periodic} if $\omega(T_\vla)$ is the union of finitely many periodic orbits, and the forward orbit of every non-singular point converges to one of them. 
\end{defn}

\subsection{Forward invariant set and periodic orbits}
\label{su:forward}

Let $ a = \|\vla\| $ and $ b = \|\vw\| $. Denote by $ K_{\vla} $ the closed ball $ B_{r}(0) $ of radius $ r=b(1+a)/(1-a)^{2} $.

\begin{lemma}  
\label{basic_lem}
Let $ z \in X $ be a non-singular point. Then there exists a sequence $ \{i_{j}\}_{j \in \Nn} $ with $ i_{j} \in \{1,\ldots,k\} $ such that for every $ n \in \Nn $, 
\begin{equation}
\label{eq:iterate}
T^{n}_{\vla}(z) = (-1) \la_{i_{1}} \cdots \la_{i_{n}} + \sum^{n}_{j=1} (-1)^{n-j}(1+\la_{i_{j}}) \la_{i_{j+1}} \cdots \la_{i_{n}} w_{i_{j}}.
\end{equation}
\end{lemma}

The proof of the lemma is straightforward, and so we omit it.

\begin{proposition} 
\label{basic_prop} 
We have $T_\vla(K_{\vla}) \subset K_{\vla} $ and $ \omega(T_{\vla}) \subset K_{\vla} $.
\end{proposition}

\begin{proof} 
The claim follows easily from the inequality  
\[
\left\|T^{n}_{\vla}(z)\right\| \le a^n\|z\| + \frac{(1+a)}{1-a} b \qquad \text{for } z \in X \setminus \SSS_{n} \text{ and } n \in \Nn,
\]
which is a direct consequence of~\eqref{eq:iterate}.
\end{proof}

%\begin{proof}
%The first part of the proposition follows easily from Lemma~\ref{basic_lem}. Next, for every non-singular $z \in X$ and every $n\in\Nn$, there exist
%$v_1,\dots,v_n\in\{w_1,\dots,w_k\}$ and $\rho_1,\dots,\rho_n\in\{\la_1,\dots,\la_k\}$ such that 
%\[
%T_{\vla}^n(z)=R_{\rho_n,v_n} \circ \cdots \circ R_{\rho_1,v_1}(z).
%\]
%By Lemma~\ref{basic_lem}, for every $\ve>0$, there exists $n_0 \ge 0 $ such that $ |T_{\vla}^n(z)| \le b(1+a)/(1-a)+\ve $ for all $n \ge n_0$. 
%\end{proof} 

By the previous proposition, every periodic orbit of $ T_{\vla} $ is contained in $ K_{\vla} $.

\begin{proposition}      
\label{finite_cor}
The map $ T_{\vla} $ has finitely many periodic orbits of period less than $ n \in \Nn $.
\end{proposition}

\begin{proof}
Suppose that $ z $ is a periodic point of period $ n $. Then by Lemma~\ref{basic_lem}, there exists a sequence $ i_{1},\ldots,i_{n} $ of integers contained in $ \{1,\ldots,k\} $ such that 
\[ 
z=T_{\vla}^n(z)=(-1)^n \la_{i_{1}} \cdots \la_{i_{n}} z + \sum_{j=1}^n(-1)^{n-j}(1+\la_{i_{j}}) \la_{i_{j+1}} \cdots \la_{i_{n}} v_j.
\]
Since $ 0 < \la_{i_{1}} \cdots \la_{i_{n}} < 1 $, we have
\begin{equation}    
\label{periodic_point_eq}
z=\frac{\sum_{j=1}^n(-1)^{n-j}(1+\la_{i_{j}}) \la_{i_{j+1}} \cdots \la_{i_{n}} v_j}{1+(-1)^{n+1} \la_{i_{1}} \cdots \la_{i_{n}}}.
\end{equation}
Thus, $ z $ is determined by the sequence $ i_{1},\ldots,i_{n} $. Since there are at most $k^n$ of these sequences, we conclude that there are at most $k^n$ periodic points of period $n$. 
\end{proof}

The simplest periodic orbit of the outer billiard about a convex $k$-gon has rotation number\footnote{Given a periodic orbit $\gamma$ of $T_\vla$, let $\Gamma$ be the closed curve obtained by joining consecutive points of $\gamma$ using straight lines. The \textit{rotation number} of $\gamma$ is defined to be the winding number of $\Gamma$ divided by $k$.} $1/k$. The sequence of the vertices of $ P $ visited by this orbit is the infinite periodic sequence $ w_1,w_{2},\ldots, w_k, w_{1},w_{2}, \ldots $. This orbit is sometimes called {\em Fagnano\footnote{Giovanni Francesco Fagnano, italian mathematician lived between the years 1715 -- 1797.} orbit}~\cite{Ta99}.

\begin{corollary}      
\label{fagnano_cor}
There exists $0<\de<1$ depending only on the vertices of $ P $ such that if $\|\vla\|<\de$, then the Fagnano orbit exists, and is the global attractor for $T_{\vla}$.
\end{corollary}

\begin{proof}
Suppose that $v_1,\ldots,v_k$ is a Fagnano orbit of $P$. Relabeling the vectors $v_i$, if necessary, and setting $ v_{k+1} = v_{1} $, we can write
\[
w_i = \frac{\la_i}{1+\la_i}v_i+\frac{1}{1+\la_i}v_{i+1}\ \quad \text{for } i=1,\ldots,k.
\]
%\begin{align*}
%w_1&=\frac{\la_1}{1+\la_1}v_1+\frac{1}{1+\la_1}v_2\,,\\
%w_2&=\frac{\la_2}{1+\la_2}v_2+\frac{1}{1+\la_2}v_3\,,\\
%& \vdots\\
%w_k&=\frac{\la_k}{1+\la_k}v_{k}+\frac{1}{1+\la_k}v_1\,.
%\end{align*}
The previous equations define a linear map $(v_1,\ldots,v_n)\mapsto (w_1,\ldots,w_n)$ which is invertible provided $\la_1\cdots\la_k\neq (-1)^{k} $. Since this condition is sufficient and necessary for the existence of the Fagnano orbit, the Fagnano orbit exists for dissipative outer billiards. 

Every periodic orbit of a dissipative outer billiard is locally attracting. Thus, to prove that the Fagnano orbit is a global attractor for $T_\vla$, it suffices to show that if $ \|\vla\| $ is sufficiently small, then every non-singular point of $ K_{\vla} $ eventually enters a small neighborhood of the Fagnano orbit. This property is a direct consequence of the following fact: for every non-singular $ z \in K_{\vla} \cap A_{i} $ with $ i \in \{1,\ldots,k\} $, we have 
\[
\|T_{\vla}(z)-v_{i+1}\| = \|T_{\vla}(z)-T_{\vla}(v_{i})\| = |\la_{i}| \|z-v_{i}\|\leq \|\vla\|\diam(K_{\vla}).
\]
%Fix $0<\delta_0<1$. By Proposition~\ref{basic_prop}, we have $\omega(T_\vla)\subset B_r(0)$ for every $\|\vla\|<\delta_0$ where $r=b(1+\delta_0)/(1-\delta_0)$. Moreover, there exists $\epsilon>0$ such that for every $\|\vla\|<\delta_0$, the open disc $B_\epsilon(v_i)$ is a local stable manifold of $v_i$, i.e. $T_\vla(B_\epsilon(v_i))\subset B_\epsilon(v_{i+1})$. Then, for every $\|\vla\|<\delta_0$ and $z\in B_r(0)$,
%$$
%|R_{\la_i,w_i}(z)-v_i|=|\la_i(w_i-z)+w_i-v_i|\leq 4r\|\vla\|<\epsilon\,,
%$$
%provided $\|\vla\|<\delta:=\epsilon/(4r)$. Thus, if $\|\vla\|<\delta$, then $T_{\vla}$ maps the disc $B_{0}(r)$ into the basin of attraction of the Fagnano orbit.
\end{proof}

\subsection{Skew-product}
\label{su:sk}
In this subsection, we assume that $ P $ is a regular $ k $-gon, and that all the components of $ \vla $ are equal to some $ \la \in (0,1) $. This situation is special,
%Remark: The skew product representation of $ T $ can be generalized to the case when $ T_{\vla} $ commutes with any symmetry of $ P $. For example when $ P $ is a $ 2n $-gon and the components of $ \vla $ are a periodic sequence $ \la_{1} \la_{2} \la_{1} \la_{2} \ldots $.]{special}, 
because the map $ T_{\vla} $ commutes with the rotation of $ \Cc $ about the center of $ P $ by an angle $ 2\pi/k $. This fact allows us to conjugate $ T_{\vla} $ to a skew product transformation. Since $ \vla $ is fixed in this subsection, we will drop the label $ \vla $ from our notation. 

Denote by $ R $ the clockwise rotation of $ \Cc $ about the center of $ P $ by an angle $ 2\pi/k $, and define $ \pi \colon \bigcup_{i} A_{i} \to A_{0} $ by $ \pi(z) = R^{i}(z) $ if $ z \in A_{i} $. Let $ S = A_{0} \cap  T^{-1}(\SSS) $. 

Next, let $ f \colon A_{0} \setminus S \to A_{0} $ be the map defined by $ f(z) = \pi \circ T(z) $ for $ z \in A_{0} \setminus S $. It is easy to see that $ A_{0} \setminus S $ consists of finitely many connected open sets $ D_{1},\ldots,D_{[k/2]+1} $ with boundary contained in $ S $, and that $ f $ is an affine transformation on each $ D_{i} $. Note that the sets $ D_{i} $ are characterized by the property $ T(D_{i}) \subset A_{i} $. Since $ f \circ \pi = \pi \circ T $ by definition, $ f $ is a factor of $ T $. For every integer $ n>0 $, the iterate $ f^{n} $ is defined only on $ A_{0} \setminus S_{n} $, where $ S_{n} $ is the set of $ z \in A_{0} $ such that $ f^{m}(z) \in S $ for some integer $ 0 \le m < n $. Let $ S_{\infty} = \bigcup_{n \ge 0} S_{n} $. Every iterate of $ f $ is defined on $ A_{0} \setminus S_{\infty} $. 

Define $ \sigma \colon A_{0} \setminus S \to \Zz_{k} $ by $ \sigma(z) = i $ whenever $ z \in D_{i} $. Then, the skew-product $ F \colon (A_{0} \setminus S) \times \Zz_{k} \to A_{0} \times \Zz_{k} $ is given by
\[
F(z,j) = (f(z),\sigma(z) + j \mod k), \qquad (z,j) \in   (A_{0} \setminus S) \times \Zz_{k}.
\]
If we set $ \sigma^{n}(z) = \sigma(f^{n-1}z) + \cdots + \sigma(z) $ for $ n \ge 1 $ and $ z \in A \setminus S_{n} $, then we have $ F^{n}(z,j)=(f^{n}(z),\sigma^{n}(z) + j \mod k) $. The transformations $ F $ and $ T $ are clearly conjugated. By abuse of notation, we also use the symbol $ \pi $ for the projection from $ A_{0} \times \Zz_{k} \to A_{0} $ given by $ \pi(z,i) = z $. 

Every periodic orbit of $ F $ (and so of $ T $) is mapped by $ \pi $ into a single periodic orbit of $ f $. The converse relation is less obvious, and is clarified in the next lemma. 

\begin{lemma}
\label{le:periodic}
Suppose that $ z \in A_{0} \setminus S_{\infty} $ is a periodic point of $ f $, i.e., $ f^{m}(z) = z $ for some integer $ m > 0 $. Let $ O_{f}(z) $ be the $ f $-orbit of $ z $. Then the set $ \pi^{-1}(O_{f}(z)) $ consists of $ \gcd(k,\sigma^{m}(z)) $ periodic orbits of $ F $ of period $ m k/\gcd(k,\sigma^{m}(z)) $. 
\end{lemma}

\begin{proof}
%Let $ x \in \pi^{-1}(O_{f}(z)) $. By the hypotheses of the lemma, there exists $ n>0 $ and $ 0 < i < k $ such that $ F^{n}(x) = (z,i) $. Since $ F $ is invertible, the previous observation implies that $ \pi^{-1}(O_{f}(z)) $ is equal to the union of the $ F $-orbits of the point $ \pi^{-1}(z) $. 

%Since $ z $ is a fixed point of $ f^{m} $, we have 
%\[ 
%F^{jm}(z,i) = (z,j\sigma^{m}(z)+i \mod k) \quad \text{for } 0 < i,j < k.
%\]
%From this expression, it is clear that every $ (z,i) $ is a periodic point of $ T $, and that its period is equal to $ nm $ with $ n $ being the smallest integer $ j>0 $ for which $ j \sigma^{m}(z) = 0 \mod k $. Such $ n $ is indeed equal to $ k / \gcd(k,\sigma^{m}(z)) $. It follows that the orbit of each $ (z,i) $ contains exactly $ n $ distinct elements of the form $ (z,j) $, and so the number of distinct periodic orbits forming $ \pi^{-1}(O_{f}(z)) $ is equal to $ k / n = \gcd(k,\sigma^{m}(z)) $.

Let $ x \in \pi^{-1}(O_{f}(z)) $. Then $ x = (z',i) $ for some $ z' \in O_{f}(z) $ and some $ 0 \le i < k $. For every integer $ j>0 $, we have 
\[ 
F^{jm}(x) = (f^{jm}(z'),j \sigma^{m}(z')+i \mod k) = (z',j \sigma^{m}(z) + i \mod k).
\]
It follows that $ x $ is a periodic point of $ F $ of period equal to 
\[
p:= m \cdot  \min \{j>0 \colon j \sigma^{m}(z) = 0 \mod k\} = mk/\gcd(k,\sigma^{m}(z)).
\] 
This shows that every element of $ \pi^{-1}(O_{f}(z)) $ is a period point of $ F $ of period $ p $. Since $ \pi^{-1}(O_{f}(z)) $ consists of $ c := mk $ elements, we conclude that $ \pi^{-1}(O_{f}(z)) $ is the union of $ c/p = \gcd(k,\sigma^{m}(z)) $ periodic orbits of $ F $.
%From this expression, it is clear that every $ (z,i) $ is a periodic point of $ F $, and that its period is equal to $ nm $ with $ n $ being the smallest positive integer for which $ n \sigma^{m}(z) = 0 \mod k $. Such $ n $ is indeed equal to $ k / \gcd(k,\sigma^{m}(z)) $. It follows that the orbit of each $ (z,i) $ contains exactly $ n $ distinct elements, and so the number of distinct periodic orbits forming $ \pi^{-1}(O_{f}(z)) $ is equal to $ k / n = \gcd(k,\sigma^{m}(z)) $.
\end{proof}

\section{Triangular billiards}   
\label{triangle}

In this section we study the dissipative outer billiard about a triangle. We consider the special case $ \la_{1} = \la_{2} = \la_{3} = \la \in (0,1) $. By affine equivariance, we do not lose generality by restricting our analysis to the equilateral triangle $\De$ of
unit side length. 

%Let $(0,0),(1,0),(\frac12,-\frac{\sqrt{3}}{2})$
%be the corners. We denote by $u,v$ the vectors whose endpoints are
%$(1,0),(-\frac12,\frac{\sqrt{3}}{2})$ respectively or in complex
%coordinates $1,e^{\frac{2\pi}{3} i}$. It is convenient to use them
%as our basis (as opposed to the standard basis $1,i\in\Cc$). The
%corners of $\De$ are $0,u,-v$. 

Let $v=e^{\frac{2\pi}{3} i}$. The corners of $\De$ are $0,1,-v$. Let $T_\la$ be the outer billiard
map on $\De$. Hereafter, we shall omit $\la$ from our notation. Let $A_0$, $A_1$ and $A_2$ be the three cones for $T$. The
apex of $A_0$ is the origin and its sides are the rays $\{a:a\geq0\}$ and $\{bv:b\ge
0\}$. The
cones $A_1$ and $A_2$ are determined by symmetry. 

Recall from Subsection~\ref{su:sk} that $S$ is the set of points in $A_0$ that are mapped by $T$ to the singular set $\SSS$, i.e. $S=A_0\cap T^{-1}(\SSS)$. For the equilateral triangle, $S$ is the ray $\{\la^{-1}v+a (1+v)\,:\, a\geq0\}$.

Note that, $T(A_0)\subset
A_1\cup A_2$. Let $A=A_0$ and let $C=T^{-1}(T(A_0)\cap A_1)$ and $E=T^{-1}(T(A_0)\cap A_2)$. A simple computation shows that
$C$ is the polygonal region in $A$ with one finite side
$[0,\la^{-1}v]$ and two infinite sides: $S$ and the positive real axis. Clearly, $E=A\setminus\overline{C}$. Note that the sets $C$ and $E$ coincide with the sets $D_{1}$ and $D_{2}$ introduced in Subsection~\ref{su:sk}. These sets are depicted in Fig.~\ref{fig:triang1}.
%
% i.e.
%$$
%D_{1,\la}=\{a+b(1+v)\colon a>-\la^{-1}\,,\,b>0\,\text{ and }\, a+b>0\}\,.
%$$
%Since $D_{2,\la}=A_0\setminus \overline{D_{1,\la}}$, we conclude that,
%$$
%D_{2,\la}=\{av+b(1+v)\colon a>\la^{-1}\,\text{ and }\, b>0\}\,.
%$$
%These sets are depicted in Fig.~\ref{fig:triang1}.
\begin{figure}[t]
\begin{center}
\includegraphics[width=1.8in]{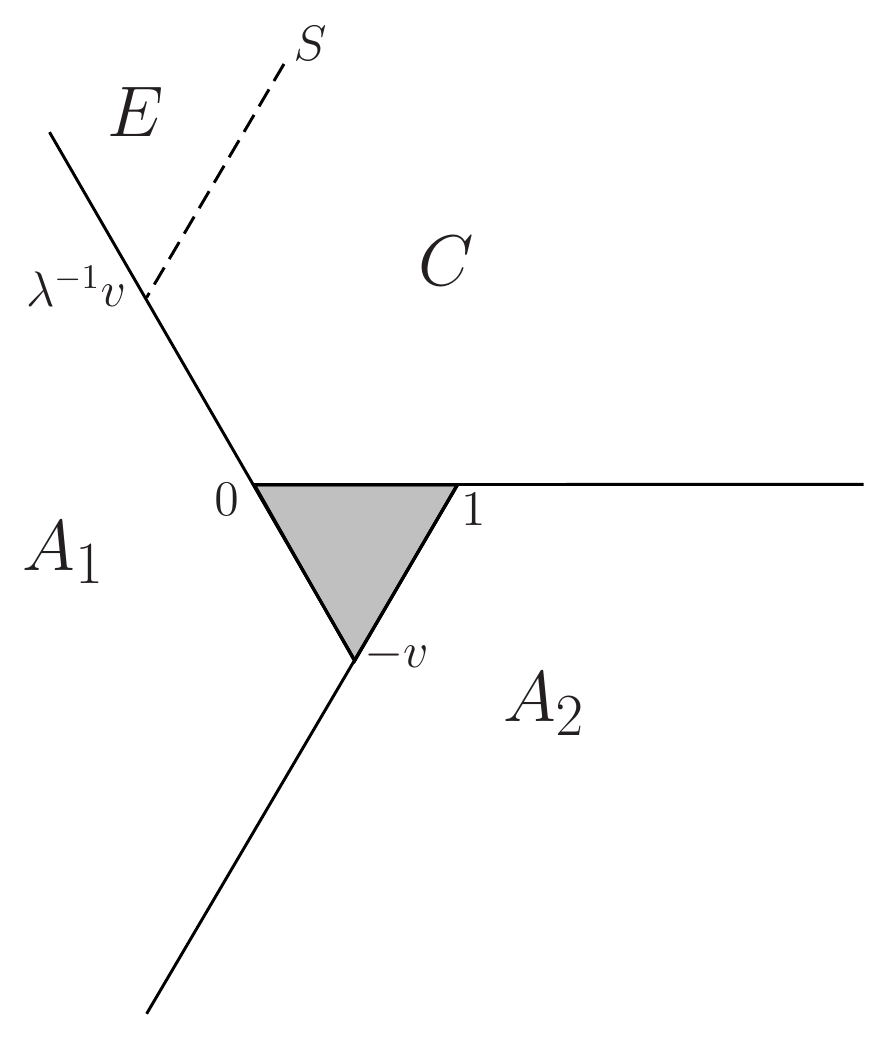}
\end{center}
\caption{Dissipative outer billiard about the equilateral triangle}
\label{fig:triang1}
\end{figure}

We now describe the map $f$. Rotating properly the sets $T(C)$ and $T(E)$ about the center of the
triangle $\De$ we can reduce $T$ to a piecewise affine map
$f \colon A \setminus S\to A$ whose defining regions are
$C$ and $E$. Let $f_C=f|_{C}$ and $f_E=f|_{E}$. It is easy to see that
$$
f_C(z)=1+\la\, (1+v)z\,,\quad f_E(z)=-v - \la\, v z\,.
$$
%$$
%f_\lambda(z)=\begin{cases}1+\la\, (1+v)z\, & {\rm if }\,z\in D_{1,\la}\,,\\
%-v - \la\, v z\, & {\rm if }\,z\in D_{2,\la}\,.
%\end{cases}
%$$

Note that $f(E)=\{a+b(1+v)\colon a>1\text{ and } b>0\}\subset C$.
Indeed, since any $z\in E$ can be
written in the form $z=\la^{-1} v + a v + b (1+v)$ for some
$a,b>0$, we have
\begin{equation*}
\begin{split}
f_E(z)& =-v - \la\, v(\la^{-1} v + a v + b (1+v))\\
&=1+\la a (1+v)+\la b\,.
\end{split}
\end{equation*}
Thus $f(E)\subset C$. Also note that, $C$ and $E$ depend on $\la$ while
$f(C)$ and $f(E)$ do not. 

%Let $\mathcal{S}_{n0}$ be the set of points in $A_0$ that are mapped to $\mathcal{S}_{10}$
%under $n-1$ iterations of $F$. Since $F$ is a piecewise affine contraction,
%the set $\mathcal{S}_{n0}$ is a union of finitely many lines.

%Let $B$ be the
%polygonal region with corners $0,u,\la^{-1}v$ and two infinite
%sides parallel to $u+v$. More precisely, $B$ is the set of points
%$au + b(u+v)$ such that $-\la^{-1}< a< 1$, $b> 0$ and
%$a+b>0$. Note that $A_0\setminus \overline{B}=D_2\cup f(D_2)$. 

Let $B$ be the set of points $a+b(1+v)\in C$ such that $0<b<\la^{-1}+\la^{-2}$. Also, let $C_1=B$ and $C_n$, $n\geq2$ be the set of points $a + b
(1+v)\in C$ such that
$$
\la^{-1}+\cdots+\la^{-2n+2}< b< \la^{-1}+\cdots+\la^{-2n}\,.
$$
The sets $B$ and $C_n$ are depicted in Fig.~\ref{fig:triang2}. 
\begin{figure}[b]
  \begin{center}
  \includegraphics[width=4in]{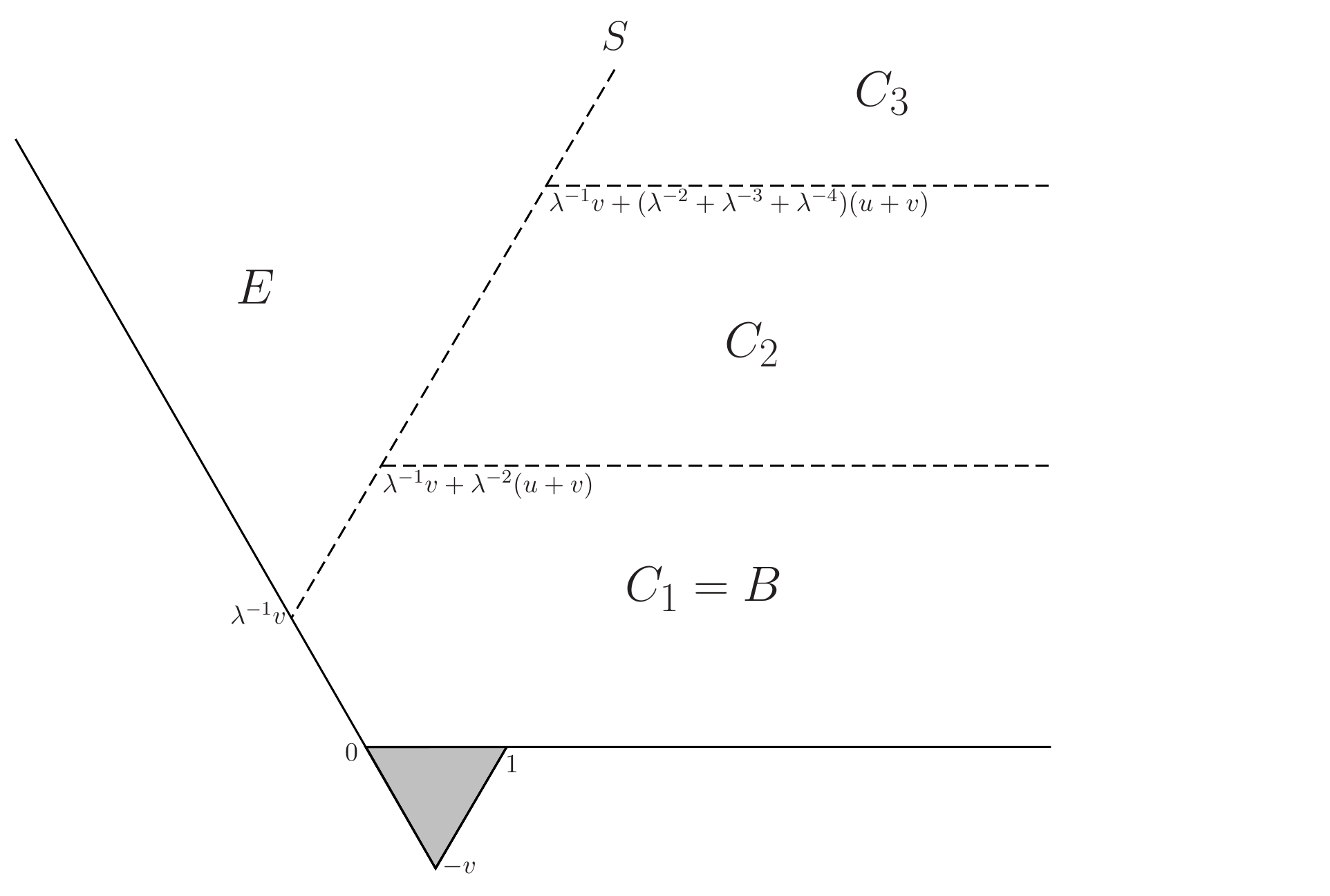}
  \end{center}
  \caption{Sets $B$ and $C_n$.}
  \label{fig:triang2}
\end{figure}

%\begin{lemma} $F(D_0)\subset B $ and $F^2( D_n )\subset D_{n-1}$ for every $n=1,2,3,\ldots$ \end{lemma}

\begin{lemma}\label{lem:BCn}
$f^2(C_n)\subset C_{n-1}$ for every integer $n\geq2$.
\end{lemma}

\begin{proof} 
Let $z\in C_n$ for some integer $n\geq2$. Then $z=a+b(1+v)$ with $a>-\la^{-1}$ and $b>\la^{-1}+\la^{-2}$. Thus
\begin{equation*}
f(z)=f_C(z)=(\la b -1)v+(1+a\la)(1+v)\,.
\end{equation*}
Since, $\la b -1>\la^{-1}$ and $1+a\la>0$ we conclude that $f(z)\in E$. Iterating again we get, $$f^2 (z)=f_E\circ f_C(z)= \alpha + \beta (1+v)$$ where $\alpha=1+\la+a\la^2$ and $\beta=\la^2b-\la-1$. Taking into account that $a>-\la^{-1}$ and $
\la^{-1}+\cdots+\la^{-2n+2}< b< \la^{-1}+\cdots+\la^{-2n}$ we get that
$$
\alpha>1\quad\text{and}\quad \la^{-1}+\cdots+\la^{-2n+4}<\beta<\la^{-1}+\cdots+\la^{-2n+2}\,.
$$
Thus, if $z\in C_n$ then $f^2(z)\in C_{n-1}$.
\end{proof}

Now we study the images of $B$ under iteration of $f$. On can easily show that $f(B)\subset C$ (a consequence of Lemma~\ref{lem:Bn} below). Let $B_n$, $n\geq1$ be the points in $B$ that are mapped to $C_n$ after one iteration of $f$, i.e. $$B_n=f^{-1}(B)\cap C_n\,,\quad n\geq1\,.$$

\begin{lemma}\label{lem:Bn} For every integer $n\geq1$ the set $B_n$ consists of points $a+b(1+v)\in B$ such that $$\la^{-2}+\cdots+\la^{-2n+1}<a+b<\la^{-2}+\cdots+\la^{-2n-1}\,.$$
\end{lemma}

\begin{proof}
Let $z=a+b(1+v)\in B$. Then $f_C(z)=\alpha+\beta(1+v)$ where
$\alpha=1-\la b$ and $\beta=\la(a+b)$. From $a>-\la^{-1}$, $a+b>0$ and $0<b<\la^{-1}+\la^{-2}$ we conclude that 
$$
f_C(B)=\left\{\alpha+\beta(1+v)\colon-\la^{-1}<\alpha<1\,,\,\beta>0\,\text{ and }\,\alpha+\beta>0\right\}\,.
$$
Thus, $f(B)\subset C$ and $f(B)\cap C_n\neq\emptyset$ for every $n\geq1$. Now $f(z)\in C_n$ if and only if $\la^{-1}+\cdots+\la^{-2n+2}< \beta< \la^{-1}+\cdots+\la^{-2n}$. Taking into account that $\beta=\la(a+b)$ we get the desired inequality.
\end{proof}

\begin{figure}[t]
  \begin{center}
  \includegraphics[width=3.5in]{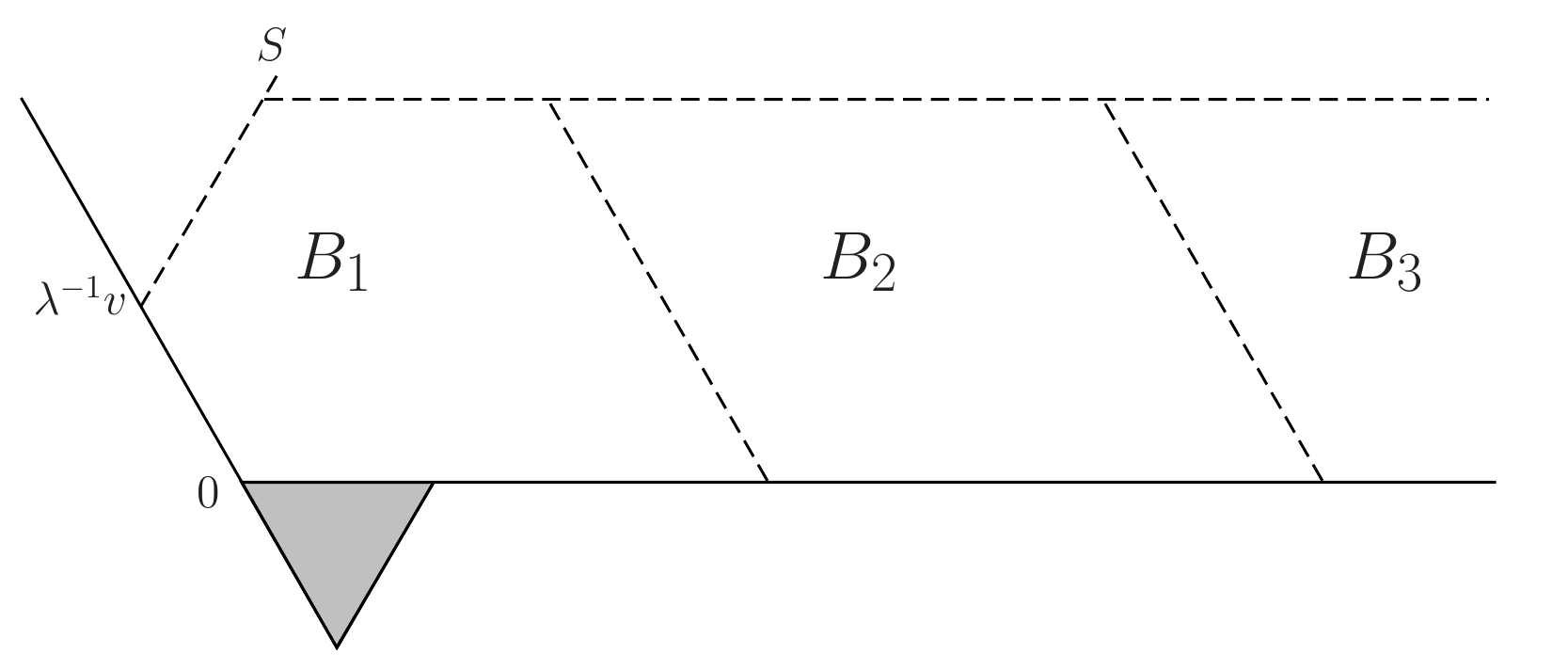}
  \end{center}
  \caption{The sets $B_n$.}
  \label{fig:triang3}
\end{figure}
The sets $B_n$ are depicted in Fig.~\ref{fig:triang3}. Let $B'=\bigcup_{n\geq1} B_n$. By previous considerations, we obtain that every point in $B'$ comes back to $B$ after an odd number of iterations of $f$. In fact, for every $n\geq1$ we have
$$
B_n=\{z\in B' \colon  f^{2n-1}(z)\in B\,\text{and}\,f^k(z)\notin B\,\text{if}\, 1\leq k< 2n-1\}\,.
$$
%Let $B^{(n)}=\cup_{i=0}^n B_i$. 
On $B'$ we define the Poincar\'e map $\phi \colon B'\to B$, i.e. the first return map to $B$. Let $\phi_n$ denote the restriction of $\phi$ to $B_n$.

\begin{lemma} \label{lem:phi}
For every integer $n\geq1$
$$
\phi_n(z)=\la^{2n-2}-v(1+\cdots+\la^{2n-3})+\la^{2n-1}(1+v)z.
$$
\end{lemma}

\begin{proof}
By Lemma~\ref{lem:BCn}, $\phi_n=(f_E\circ f_C)^{n-1}\circ f_C$. The proof goes by induction on $n$ and is left to the reader.
\end{proof}
%
%By
%Lemma~\ref{lemmaBn}, it is
%convenient to represent points in $B_n$ using the basis $\{u+v,v\}$
%with the origin shifted to $-v$. To simplify the notation we use
%the same letters to represent $B_n$ and $\phi_n$ in the plane $a+ib$.
%Let
%$$
%\gamma_n(\la)=1+\la+\cdots+\la^{2n-1}\quad\text{and}\quad \eta_n(\la)=1+\la^{-1}+\cdots+\la^{-2n}\,.
%$$
%Then $B_n$ is the set of points $a+ib$ such that $a>0$,
%$0<b<1+\la^{-1}$ and $\eta_n(\la)<a+ b<\eta_{n+1}(\la)$. 

In the following lemma we prove that $\phi$ has some monotonicity property. 
%Let $B^{(n)}=\bigcup_{i=1}^n B_i$.

\begin{lemma}\label{lem:Bnmonotone}
For every integer $n\geq1$ the following holds:
\begin{enumerate}
	\item $\phi(B_n)\subset B_1\cup\cdots\cup B_{n+1}$,
	\item $\phi(\phi(B_n)\cap B')\subset B_1\cup\cdots\cup B_{n}$.
\end{enumerate}
\end{lemma}

\begin{proof}
We start by proving item (1). Let $z=a+b(1+v)\in B_n$. By Lemma~\ref{lem:phi}, we have $\phi_n(z)=\alpha+\beta(1+v)$ where
\begin{equation*}\begin{split}
\alpha&=1+\cdots+\la^{2n-2}-b\la^{2n-1}\,,\\ 
\beta&=\la^{2n-1}(a+b)-1-\cdots-\la^{2n-3}\,.
\end{split}\end{equation*}
Taking into account the definition of $B_{n+1}$ we only need to show that $\alpha+\beta<\la^{-2}+\cdots+\la^{-2n-3}$. But since, $b>0$ and $a+b<\la^{-2}+\cdots+\la^{-2n-1}$ we conclude that
\begin{align*}
\alpha+\beta&=\la^{2n-2}+\la^{2n-1}(a+b)-b\la^{2n-1}\\
&<\la^{-2}+\la^{-1}+1+\cdots+\la^{2n-2}\\
&<\la^{-2}+\cdots+\la^{-2n-3}
\end{align*}
for every $\la\in(0,1)$, and the first claim follows.

To prove item (2), let $z=a+b(1+v)\in B_n\cap \phi^{-1}(B')$. We want to show that $\phi^2(z)\in B_1\cup\cdots\cup B_{n}$. By item (1), we only need to consider two cases:
\begin{enumerate}
\item[(a)] $z\in B_n\cap \phi^{-1}(B_n)$. Thus $\phi^2(z)=\phi_n^2(z)=\alpha'+\beta'(1+v)$ where $\alpha'+\beta'=\la^{2n-2}(1+\la+\cdots+\la^{2n-1}-b\la^{2n})$. Since $b>0$ we get
$$
\alpha'+\beta'<\la^{-2}+\cdots+\la^{-2n-1}
$$
%$$
%\alpha'+\beta'<\la^{2n-2}(1+\cdots+\la^{2n-1})<\la^{-2}+\cdots+\la^{-2n-1}
%$$
for every $\la\in(0,1)$. Thus (2) follows.
\item[(b)] $z\in B_n\cap \phi^{-1}(B_{n+1})$. In this case $\phi^2(z)=\phi_{n+1}\circ\phi_n(z)$. As before we have $\phi_{n+1}\circ\phi_n(z)=\alpha'+\beta'(1+v)$ where
$\alpha'+\beta'=\la^{2n}(1+\la+\cdots+\la^{2n-1}-b\la^{2n})$.
Taking into account that $b>0$ we get
$$
\alpha'+\beta'<\la^{-2}+\cdots+\la^{-2n-1}
$$
%$$
%\alpha'+\beta'<\la^{2n}(1+\cdots+\la^{2n-1})<\la^{-2}+\cdots+\la^{-2n-1}
%$$
for every $\la\in(0,1)$. Thus (2) follows again in this case. 
\end{enumerate}
\end{proof}

\begin{lemma}\label{lem:phiN}
For every $\la\in(0,1)$ there exists an integer $N=N(\la)>0$ such that $\phi(B')\subset B_1\cup\cdots\cup B_{N}$. 
\end{lemma}

\begin{proof}
The claim follows from the fact that, 
$$
\Re(\phi(z))<1+\cdots+\la^{2n}<\frac{1}{1-\la}\quad\forall\,z\in B_n\,.
$$
\end{proof}

Let $S_{\infty}$ be the set of points in $B$ whose orbit eventually intersect a discontinuity point of $\phi$. Let $\psi_n$ denote the affine extension of $\phi_n$ to the complex plane and denote by $z_n$ the fixed point of the extension $\psi_n$. Also denote by $w_n$ the unique fixed point of the map $\psi_{n+1}\circ\psi_n$. Note that $u_n:=\psi_n(w_n)$ is the unique fixed point of the map $\psi_n\circ\psi_{n+1}$. Finally, let $\psi=\{\psi_n\}$ and denote by $\text{Per}(\psi)$ the union of all $z_n$, $w_n$ and $u_n$. 

\begin{proposition}\label{prop:attractor1}
For every $z\in B\setminus S_{\infty}$ there exists an integer $n\geq1$ and $z_*\in\text{Per}(\psi)$ such that $\lim_{k\to\infty}\phi^{2k}(z)=z_*\in \overline{B_n}$.
\end{proposition}

\begin{proof}
Given $z\in B\setminus S_\infty$, by Lemma~\ref{lem:Bnmonotone}, there exist integers $k_0\geq0$ and $n\geq1$ such that $\phi^{2k}(z)\in B_n$ for every $k\geq k_0$. Let $w=\phi^{2k_0}(z)$. Thus $\phi^{2k}(w)\in B_n$ for every $k\geq0$. Again, by Lemma~\ref{lem:Bnmonotone}, $$\phi^{2k+1}(w)\in B_{n-1}\cup B_n \cup B_{n+1}$$ for every $k\geq0$. Thus we have three cases:
\begin{enumerate}
	\item There exists $k_1\geq0$ such that $\phi^{2k_1+1}(w)\in B_{n-1}$. Then, by Lemma~\ref{lem:Bnmonotone} we have that $\phi^{2k+1}(w)\in B_{n-1}$ for every $k\geq k_1$. Thus $\phi^{2k}(w)=(\phi_{n-1}\circ\phi_n)^k(w)$ for every $k\geq k_1$, which implies that $\lim_{k\to\infty}\phi^{2k}(z)= u_{n-1}$.
	\item There exists $k_1\geq0$ such that $\phi^{2k_1+1}(w)\in B_{n}$ and $\phi^{2k+1}(w)\notin B_{n-1}$ for every $k\geq0$. Then, again by Lemma~\ref{lem:Bnmonotone} we have that $\phi^{2k+1}(w)\in B_{n}$ for every $k\geq k_1$. Thus $\phi^{2k}(w)=\phi_n^{2k}(w)$ for every $k\geq k_1$, which implies that $\lim_{k\to\infty}\phi^{2k}(z)= z_{n}$.
	\item $\phi^{2k+1}(w)\in B_{n+1}$ for every $k\geq0$. Then, $\phi^{2k}(w)=(\phi_{n+1}\circ\phi_n)^k(w)$ for every $k\geq 0$, which implies that $\lim_{k\to\infty}\phi^{2k}(z)= w_{n}$.
\end{enumerate}
\end{proof}

Using the terminology introduced in Section~\ref{intro}, we conclude that $ \phi $ is asymptotically periodic. Let $ \omega(\phi) $ denote the $ \omega $-limit set of $ \phi $.

\begin{corollary}
For every $0<\lambda<1$, let $N=N(\la)$ be the integer in Lemma~\ref{lem:phiN}. Then
$\omega(\phi)\subset\{z_1,w_1,u_1,\ldots,z_N,w_N,u_N\}$.
\end{corollary}

\begin{proof}
The claim follows from Lemma~\ref{lem:phiN} and Proposition~\ref{prop:attractor1}.
\end{proof}

Next, we perform a bifurcation analysis of the $\omega$-limit set of $\phi$. 

%By Lemma~\ref{lem:phi}, the map $\psi_n$ reads,
%$$
%\psi_n(a+b(1+v))=\gamma_{n+1}(\la)-\la^{2n+1}b+i(\la^{2n+1}(a+b)-\gamma_{n+1}(\la)+1)\,.
%$$
%
%\begin{example}
%When $n=0$ we have $$\psi_0(a+ib)=1+\la -\la b+i(\la (a+b)-\la)\,.$$
%The fixed point of $\psi_0$ is $z_0=a+ib$ where 
%$$
%a=\frac{1}{1-\lambda +\lambda ^2}\quad\mathrm{and}\quad b=\frac{\lambda ^2}{1-\lambda +\lambda ^2}\,.
%$$
%Thus, $\phi_0$ has a fixed point if and only if $z_0\in B_0$, i.e.
%\begin{equation*}
%a>0,\quad 0<b<1+\la^{-1}\quad\mathrm{and}\quad
%1<a+b<1+\la^{-1}+\la^{-2}\,.
%\end{equation*}
%A simple computation shows that the above inequalities hold for every $0<\la<1$. Thus, $\phi_0$ has a fixed point for every $0<\la<1$.
%\end{example}
%
%Below we generalize the previous example to any $n\geq1$.
%\begin{lemma} \label{lem:psin} Let $z=a+b(1+v)$ where $\la^{-2}+\cdots+\la^{-2n-1}<a+b<\la^{-2}+\cdots+\la^{-2n-3}$. Then $\psi_n(z)=\alpha+\beta(1+v)$ where $\alpha\in\Rr$ and $0<\beta<\la^{-1}+\la^{-2}$.
%\end{lemma}

\begin{lemma} \label{lem:psin} Let $z=a+b(1+v)$ such that $z=\psi_n(z)$. Then $z\in B_n$ if and only if $\la^{-2}+\cdots+\la^{-2n+1}<a+b<\la^{-2}+\cdots+\la^{-2n-1}$. 
\end{lemma}

\begin{proof}
By Lemma~\ref{lem:phi}, $\psi_n(z)=a+b(1+v)$ where $b=\la^{2n-1}(a+b)-1-\cdots-\la^{2n-3}$. Thus $0<b<\la^{-1}+\la^{-2}$ if and only if $\la^{-2}+\cdots+\la^{-2n+1}<a+b<\la^{-2}+\cdots+\la^{-2n-1}$. The claim follows.
\end{proof}

%\begin{lemma}\label{lem:zntrig}$z_n=a_n+b_n(1+v)$ where
%\begin{align*}
%a_n&=\frac{1-\lambda ^{2 n-1}-\lambda ^{4 n-3}+\lambda ^{4 n-2}}{(1-\lambda ) \left(1-\lambda ^{2 n-1}+\lambda ^{4 n-2}\right)}\,,\\
%b_n&=-\frac{1-\lambda ^{2 n-2}-\lambda ^{2 n-1}+\lambda ^{4 n-2}}{(1-\lambda ) \left(1-\lambda ^{2 n-1}+\lambda ^{4 n-2}\right)}\,.
%\end{align*}
%\end{lemma}
%\begin{proof}
%The claim follows from Lemma~\ref{lem:phi}.
%\end{proof}
%
%\begin{lemma} \label{lem:wntrig}
%$w_n=a_n+b_n(1+v)$ where
%\begin{align*}
%a_n&=\frac{1-\lambda ^{4 n-1}+\lambda ^{4 n}-\lambda ^{6 n}-\lambda ^{8 n-1}+\lambda ^{8 n}}{(1-\lambda ) \left(1+\lambda ^{4 n}+\lambda ^{8 n}\right)}\,,\\
%b_n&=-\frac{1-\lambda ^{2 n}-\lambda ^{4 n-1}+\lambda ^{4 n}-\lambda ^{6 n}+\lambda ^{8 n}}{(1-\lambda ) \left(1+\lambda ^{4 n}+\lambda ^{8 n}\right)}\,.
%\end{align*}
%\end{lemma}
%\begin{proof}
%The claim follows from Lemma~\ref{lem:phi}.
%\end{proof}

\begin{proposition}\label{lem:fxpts}
$z_n$ is a fixed point of $\phi_n$ if and only if $\la\in(\la_{2n-1},1)$ where the sequence $\{\la_m\}$ is defined in the Appendix.
\end{proposition}

\begin{proof}
We want to show that $z_n\in B_n$ if and only if $\la\in(\la_{2n-1},1)$. By Lemma~\ref{lem:phi}, $z_n=a_n+b_n(1+v)$ where 
$$
a_n+b_n=\frac{\lambda ^{2 n-2}-\lambda ^{4 n-3}}{(1-\lambda ) \left(1-\lambda ^{2 n-1}+\lambda ^{4 n-2}\right)}\,.
$$
By Lemma~\ref{lem:psin}, $z_n\in B_n$ if and only if
$$
\la^{-2}+\cdots+\la^{-2n+1}<a_n+b_n<\la^{-2}+\cdots+\la^{-2n-1}\,.
$$
Substituting the expression for $a_n+b_n$ into the previous inequalities we get
$$
\la^{2}(1-\la^{2n-2})<\frac{\lambda ^{4 n-1} \left(1-\lambda ^{2 n-1}\right)}{1-\lambda ^{2 n-1}+\lambda ^{4 n-2}}<1-\la^{2n}\,.
$$
We study each inequality separately. The right-hand side inequality is equivalent to $$1-\la^{2n-1}-\la^{2n}+\la^{4n-2}>0\,,$$ which holds for every $\la\in(0,1)$. Indeed,
$$
1-\la^{2n-1}-\la^{2n}+\la^{4n-2}>(1-\la^{2n-1})^2>0\,.
$$
The left-hand side inequality is equivalent to 
$$
p_{2n-1}(\la)=1-\lambda ^{2 n-2}-\lambda ^{2 n-1}+\lambda ^{4 n-2}<0\,.
$$
By Lemma~\ref{le:solutionQ}, we conclude that $p_{2n-1}(\la)<0$ for every $\la\in (\la_{2n-1},1)$. This concludes the proof of the proposition. 
\end{proof}

%\begin{proposition}
%The Poincar\&#039;e map has a unique periodic point of period two in $K_n$
%if and only if $\la\in(\alpha_n,1)$ where $\alpha_n$ is the unique solution in $[0,1)$ of the equation
%$$
%1-\lambda ^{2+2 n}-\lambda ^{3+4 n}+\lambda ^{4+4 n}-\lambda ^{6+6 n}+\lambda ^{8+8 n}=0\,.
%$$
%\end{proposition}

\begin{proposition}\label{prop:per2}
$w_n$ is a fixed point of $\phi_{n+1}\circ\phi_n$ if and only if $\la\in(\gamma_{2n-1},1)$ where the sequence $\{\gamma_m\}$ is defined in the Appendix.
\end{proposition}
\begin{proof}
We want to show that $w_n\in B_n$ and $u_n\in B_{n+1}$ if and only if $\la\in(\gamma_{2n-1},1)$. By Lemma~\ref{lem:phi}, $w_n=a_n+b_n(1+v)$ and $u_n=\psi_n(w_n)=\alpha_n+\beta_n(1+v)$ where 
$$
a_n+b_n=\frac{\lambda ^{2 n} \left(1-\lambda ^{6 n-1}\right)}{(1-\lambda ) \left(1+\lambda ^{4 n}+\lambda ^{8 n}\right)}\,,\quad
\alpha_n+\beta_n=\la^{2n-1}a_n+\la^{2n-2}\,.
$$
By Lemma~\ref{lem:psin}, $w_n\in B_n$ and $u_n\in B_{n+1}$ if and only if
\begin{gather*}
\la^{-2}+\cdots+\la^{-2n+1}<a_n+b_n<\la^{-2}+\cdots+\la^{-2n-1}\,,\\
\la^{-2}+\cdots+\la^{-2n-1}<\alpha_n+\beta_n<\la^{-2}+\cdots+\la^{-2n-3}\,.
\end{gather*}
Substituting the expressions for $a_n+b_n$ and $\alpha_n+\beta_n$ into the previous inequalities we get
%\begin{align*}
%a_n+b_n&=\frac{\lambda ^{2 n+2} \left(1-\lambda ^{6 n+5}\right)}{(1-\lambda ) \left(1+\lambda ^{4 n+4}+\lambda ^{8 n+8}\right)}\,,\\
%\alpha_n+\beta_n&=\la^{2n+1}a_n+\la^{2n}\,.
%\end{align*}
\begin{gather*}
\la^{2}(1-\la^{2n-2})<\frac{\lambda ^{4 n+1} \left(1-\lambda ^{6 n-1}\right)}{ 1+\lambda ^{4 n}+\lambda ^{8 n}}<1-\la^{2n}\,,\\
\la^{2}(1-\la^{2n})<\frac{\lambda ^{4n+1} \left(1-\lambda ^{6 n+1}\right)}{1+\lambda ^{4 n}+\lambda ^{8 n}}<1-\la^{2n+2}\,.
\end{gather*}
It turns out that, among these four inequalities, there are two which are equivalent. In fact, we can reduce the above chain of inequalities to the following system:
\begin{equation*}
\begin{cases}
1-\lambda ^{2 n-2}-\lambda ^{4 n-1}+\lambda ^{4 n}-\lambda ^{6
n}+\lambda ^{8 n}<0\,,\\
1-\lambda ^{2 n}+\lambda ^{4 n}-\lambda ^{4 n+1}-\lambda ^{6 n}+\lambda ^{8 n}>0\,,\\
1-\lambda ^{2 n}-\lambda ^{4 n-1}+\lambda ^{4 n}-\lambda ^{6
n}+\lambda ^{8 n}<0\,.
\end{cases}
\end{equation*}
We claim that this system of inequalities holds for every $\la\in(\gamma_{2n-1},1)$. Let $I_1$, $I_2$ and $I_3$ denote the first, second and third polynomials in the previous system. Note that $I_1(\la)<I_3(\la)$ and $I_2(\la)>0$ for every $\la\in(0,1)$. Indeed, $I_3(\la)-I_1(\la)=\la^{2n-2}(1-\la^{2})>0$. Moreover, since $\la^{2n}+\la^{6n}>2\la^{4n}$ and
$\la^{4n}>\la^{4n+1}$ for every $\la\in(0,1)$ we get
$$
I_2(\la)>1-4\lambda ^{2 n}+6\lambda ^{4n}-4\lambda ^{6 n}+\lambda ^{8 n}=(1-\la^{2n})^8>0\,.
$$
%\begin{align*}
%R_n(\la)&=1-4\lambda ^{2 n+2}+3\lambda ^{2 n+2}+\lambda ^{4 n+4}-\lambda ^{4 n+5}+3\lambda ^{6 n+6}-4\lambda ^{6 n+6}+\lambda ^{8 n+8}\\
%&>1-4\lambda ^{2 n+2}+6\lambda ^{4n+4}-4\lambda ^{6 n+6}+\lambda ^{8 n+8}\\
%&=(1-\la^{2n+2})^8>0\,.
%\end{align*}
Finally, by Lemma~\ref{le:solutionT}, $I_3(\la)<0$ for every $\la\in(\gamma_{2n-1},1)$ and the proof of the proposition is complete.
\end{proof}

\begin{remark}
Taking into account the return time to $B$, we conclude that the fixed point of $\phi_n$ corresponds to a period orbit of period $2n-1$ of $f$ and the fixed point of $\phi_{n+1}\circ\phi_n$ corresponds to a periodic orbit of period $4n$ of $f$.
\end{remark}

%Taking into account the return time to $B$, we conclude that the fixed point of $\phi_n$ corresponds to a period orbit of period $2n+1$ of $f$ and a period orbit of period $3(2n+1)$ of $T$. Similarly, the fixed point of $\phi_{n+1}\circ\phi_n$ corresponds to a periodic orbit of period $4(n+1)$ of $f$ and a periodic orbit of period $12(n+1)$ of $T$.  Let us call these two families of periodic orbits \textit{type I} and \textit{type II}, respectively. 

For every $\la\in(0,1)$ let $m_1=m_1(\la)>0$ and $m_2=m_2(\la)>0$ be the following integers,
\begin{align*}
m_1&=\max\{n\in \Nn\colon \la_{2n-1}<\la \leq\la_{2n+1}\}\,,\\
m_2&=\max\{n\in \Nn\colon \gamma_{2n-1}<\la\leq\gamma_{2n+1}\}\,.
\end{align*}
The following result completely describes the bifurcations of the $\omega$-limit set of $\phi$.

\begin{theorem}\label{th:finaltrianglephi}
For every $\la\in(0,1)$ we have, $$\omega(\phi)=\{z_1,\ldots,z_{m_1},w_1,u_1,\ldots,w_{m_2},u_{m_2}\}\,.$$
\end{theorem}

\begin{proof}
When $\la\notin\{\la_{2n-1}\}\cup\{\gamma_{2n-1}\}$, the claim follows from Proposition~\ref{prop:attractor1}, \ref{lem:fxpts} and \ref{prop:per2}. 

So, suppose that $\la=\la_{2n-1}$ for some $n\geq2$. We want to show that $z_n\notin\omega(\phi)$. It follows from the proof of Proposition~\ref{lem:fxpts}, that $z_n=a_n+b_n(1+v)$ where $a_n=\la^{-2}+\cdots+\la^{-2n+1}$ and $b_n=0$. Thus, $z_n$ is located at the bottom left corner of $B_n$. Since, by Lemma~\ref{lem:phi}, $\psi_n$ is a composition of an anticlockwise rotation of angle $\pi/3$ and a contraction towards $z_n$, we conclude that under the iteration of $\phi_n$ every point close to $z_n$ will eventually enter $B_{n-1}$. Thus, moving away from $z_n$. This proves that $z_n\notin\omega(\phi)$. 

Now suppose that $\la=\gamma_{2n-1}$ for some $n\geq2$. As before we want to show that $w_n,u_n\notin\omega(\phi)$. It is sufficient to show that $u_n\notin\omega(\phi)$. It follows from the proof of Proposition~\ref{prop:per2}, that $u_n=\alpha_n+\beta_n(1+v)$ where $\alpha_n+\beta_n=\la^{-2}+\cdots+\la^{-2n-1}$. Hence $u_n$ belongs to the common boundary of $B_n$ and $B_{n+1}$. Since, by Lemma~\ref{lem:phi}, $\psi_n\circ\psi_{n+1}$ is composition of an anticlockwise rotation of angle $2\pi/3$ and a contraction towards $u_n$, we conclude that under the iteration of $\phi_n\circ\phi_{n+1}$ every point in $B_{n+1}$ close to $u_n$ will eventually enter $B_{n}$. Thus, by Lemma~\ref{lem:Bnmonotone}, every point close to $u_n$ will eventually enter $B_1\cup\cdots\cup B_n$ and never leave this set. This proves that $u_n\notin\omega(\phi)$ and concludes the proof of the theorem.
\end{proof}

Denote by $ \dist $ the Euclidean distance on $ \Cc $. The following result is a consequence of Theorem~\ref{th:finaltrianglephi} and Lemma~\ref{le:periodic}.

\begin{corollary}
\label{co:finaltriangle}
For every $ \la \in (0,1) $, the map $ T $ has exactly $ m_1=m_1(\la) $ distinct periodic orbits $ \Gamma_{1,i} $, $i=1,\ldots,m_1$, each one having period $ 3(2i-1) $ and exactly $m_2=m_2(\la)$ distinct periodic orbits $ \Gamma_{2,i} $, $i=1,\ldots,m_2$, each one having period $ 12i $. Moreover, for every non-singular $ z \in X $, there exist $i\in\{1,2\}$ and $ 1 \le j \le m_i $ such that $ \dist(T^{n}(z),\Gamma_{i,j}) \to 0 $ as $ n \to +\infty $. (see Fig.~\ref{fi:basintriangle}).
\end{corollary}

%\begin{corollary}
%\label{co:finaltriangle}
%For every $ \la \in (0,1) $, the orbit of any non-singular point of $ \, T_\la $ converges to a periodic point of $ \, T_\la $. The map $ T_\la $ has exactly $m+1$ periodic orbits of periods $3,9,\ldots,3(2m+1)$ and exactly $k+1$ periodic orbits of period $12,24,\ldots,12(k+1)$ (see Fig.~\ref{fi:basintriangle}).
%\end{corollary}

\begin{figure}
\begin{center}
\includegraphics[trim=0.15cm 0.15cm 0.15cm 0.15cm, clip, width=6cm]{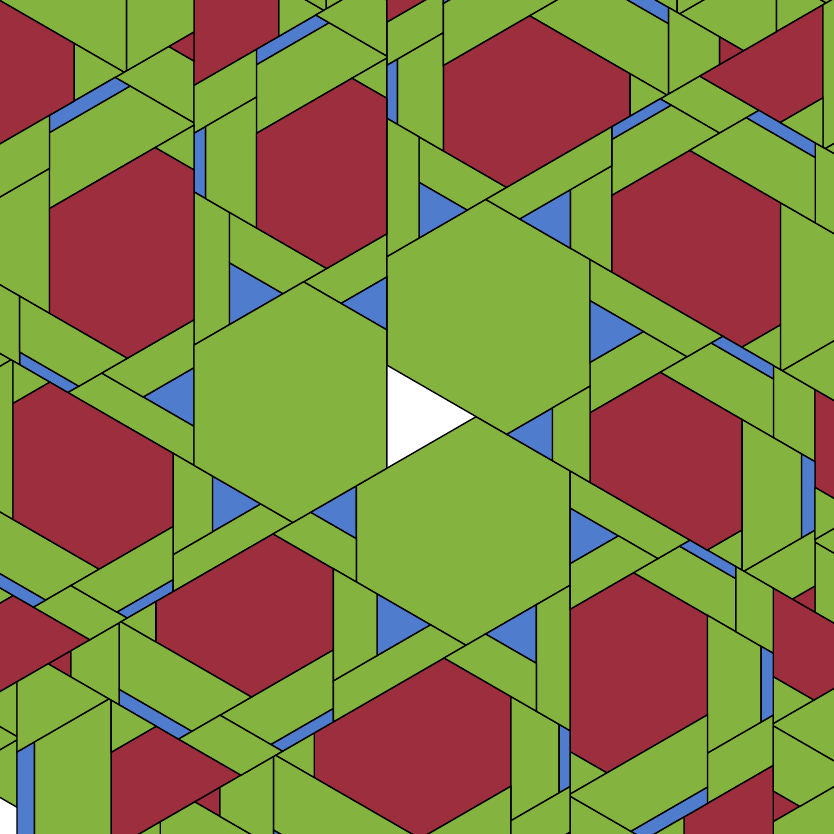}
\end{center}
\caption{Dissipative outer billiard about the equilateral triangle for $ \la = 0.95 $. The billiard map has three periodic orbits $\Gamma_{1,1}$, $\Gamma_{1,2}$ and $\Gamma_{2,1}$ of periods $3$, $9$ and $12$, respectively. The basin of attraction of each periodic orbit is the union of all the regions of the same color.}
% in the $ \omega $-limit set of the dissipative outer square. \added{Here $ \lambda = 0.95 $ and $ m = 3 $ with $ m $ as in Theorem~\ref{th:final}.}}
\label{fi:basintriangle}
\end{figure}

%\begin{proof} \added{See the new proof of Corollary~\ref{co:final}.}
%The first part of the corollary follows from Theorem~\ref{th:finaltrianglephi} and the symmetry of $T$ under the rotation $\pi$. To compute the periods of the periodic orbits note that $T$ is conjugated to the map $F(z,j)=(f(z),\sigma(z) + j \pmod{3}) $, where $\sigma(z) = i+i \pmod{3} $ if $z\in D_i$. Let $z=z_n$ with $n\leq m$. Then $z$ is a periodic point of period $2n+1$ of $f$ and $f^{2n+1}(z)=(f_2\circ f_1)^n\circ f_1(z)$. Thus $\sigma^{2n+1}(z)=3n+1 \pmod{3} $. Since $3n+1$ and $3$ are coprime it follows that $(z_n,j)$ is a periodic point of period $3(2n+1)$ of $F$, thus of $T$. By symmetry, all points $(z_n,j)$, $j=0,\ldots 2$ belong to the same period orbit. Now suppose that $z=w_n$ for some $n\leq k$. Then $z$ is a periodic point of period $4(n+1)$ of $f$ and $f^{4(n+1)}(z)=(f_2\circ f_1)^{n+1}\circ f_1\circ(f_2\circ f_1)^{n}\circ f_1(z)$. Thus $\sigma^{4n+4}(z) = 6n+5 \pmod{3} $. Since $6n+5$ and $3$ are coprime it follows that $(w_n,j)$ is a periodic point of period $12(n+1)$ of $F$, thus of $T$. Again, by symmetry, it is easy to see that all points $(w_n,j)$, $j=0,\ldots 2$ belong to the same period orbit. This concludes the proof.
%\end{proof}

\begin{proof}
Let $ 0<\la<1 $. Since $ T $ and $ F $ (see Subsection~\ref{su:sk}) are conjugated, we prove the corollary with $ T $ replaced by $ F $.

By Theorem~\ref{th:finaltrianglephi} and Lemma~\ref{le:periodic}, it follows that $$ O_{f}(z_{1}),\ldots,O_{f}(z_{m_1}),O_{f}(u_{1}),\ldots,O_{f}(u_{m_2}) $$ are the only periodic orbits of $ f $. Since $ z_{i} $ is a fixed point of $ \phi_{i} $, it follows from the definition of $ \phi_{i} $ that the period of $ O_{f}(z_{i}) $ is equal to $ 2i-1 $, and that $ f^{2i-1}(z_{i}) = (f_E\circ f_C)^{i-1}\circ f_C(z_i) $. An easy computation then shows that $ \sigma^{2i-1}(z_{i}) = 3i-2 \mod 3 $. Since $ 3i-2 $ and $ 3 $ are coprime, Lemma~\ref{le:periodic} implies that each set $ \Gamma_{1,i} := \pi^{-1}(O_{f}(z_{i})) $ is a periodic orbit of $ F $ of period $ 3(2i-1) $. 

A similar argument shows that $ \Gamma_{2,i} := \pi^{-1}(O_{f}(u_{i})) $ is a periodic orbit of $ F $ of period $ 12i $. Since $ f $ is a factor of $ F $, all periodic orbits of $ F $ are contained in the union of the sets $\Gamma_{i,j}$, $i=1,2$ and $j=1,\ldots,m_i$.

Now, let $ z $ be a non-singular point of $ F $. Theorem~\ref{th:finaltrianglephi} implies that there exists $z_*\in\text{Per}(\psi)$ such that $ \dist(f^{n}(z),O_{f}(z_{*})) \to 0 $ as $ n \to \infty $. By previous considerations, $\pi^{-1}(O_{f}(z_{*}))=\Gamma_{i,j}$ for some $i=1,2$ and $j=1,\ldots,m_i$. Finally, note that $ f $ is a factor of $ F $ and that $ \dist(w,\pi^{-1}(D)) = \dist(\pi(w),D) $ for every $ w \in X $ and every set $ D \subset A $. Combining all the previous observations together, we conclude that 
\[
\lim_{n \to +\infty} \dist(F^{n}(z),\Gamma_{i,j}) = \lim_{n \to +\infty} \dist(f_{n}(\pi(z)),O_{f}(z_*)) = 0.
\]
\end{proof}

\section{Square billiard}   
\label{square}

We now study the dissipative outer billiard about the square. As for the triangular billiard, we consider the special case $ \la_{1} = \cdots = \la_{4} = \la \in (0,1) $. In our analysis, we follow closely the approach used for the dissipative outer billiard about the equilateral triangle. In particular, we will construct a first return map $\phi$ on a strip $B$ and obtain a detail description of the dynamics of the dissipative outer billiard map $T$. For the square, the geometry of the action of the map $\phi$ is simpler, because the singular set of the billiard map consists of horizontal and vertical lines. In this section, we will formulate a result without proof, every time that the proof can be easily recovered from the corresponding ones for the equilateral triangle.

For convenience, we assume that the corners of the square are given by $w_0=0 $, $w_1=-i$, $w_2=1-i$ and $w_3=1$. Let $A_i$ be the open conical region corresponding to $w_i$ (see Fig.~\ref{square_fig2}). 
\begin{figure}[t]
\begin{center}
\includegraphics[scale=.4]{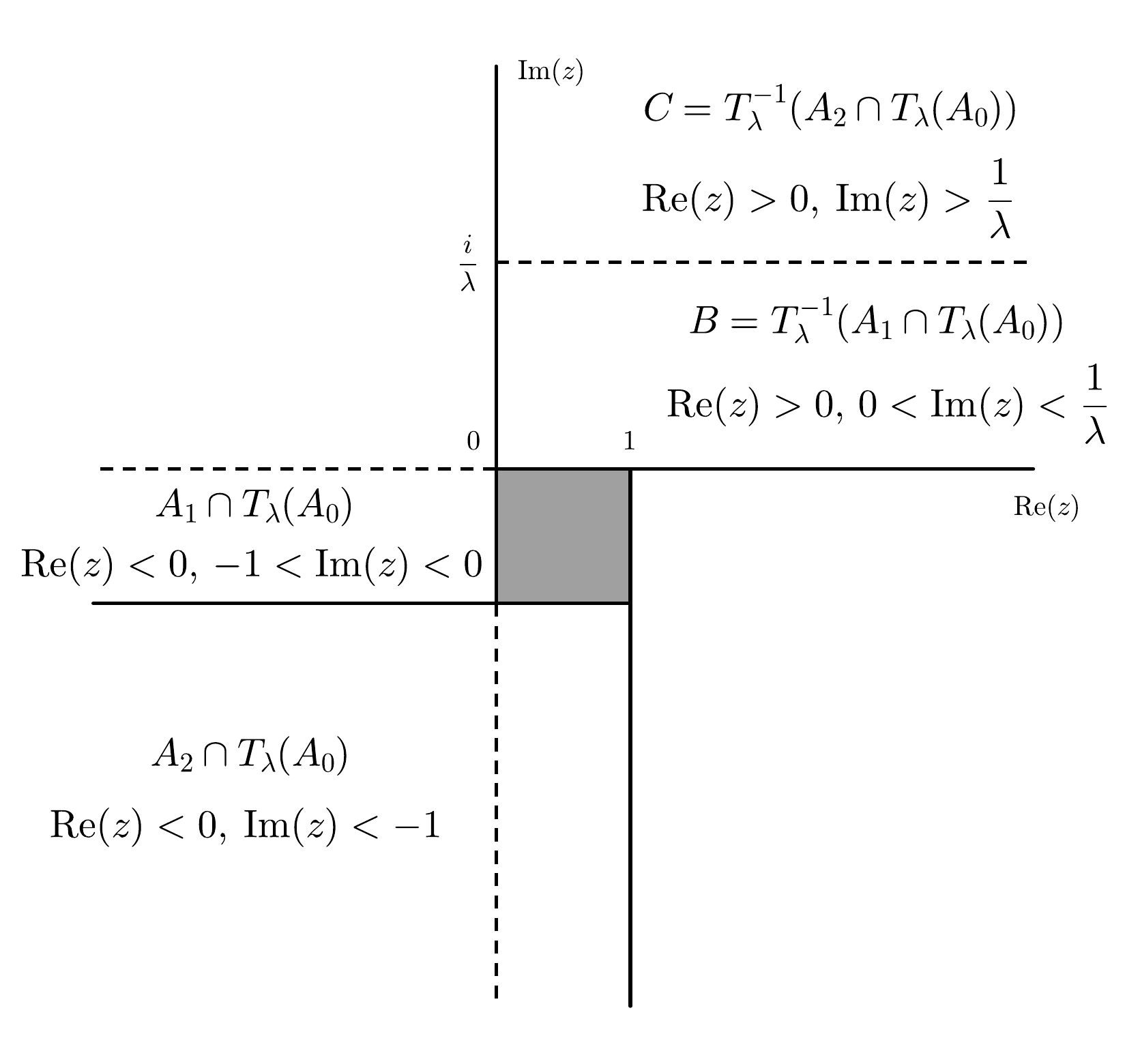}
\caption{Outer billiard map about the unit square and reduction to the first quadrant.}
\label{square_fig2}
\end{center}
\end{figure}
%\begin{figure}[t]
%\begin{center}
%\includegraphics[scale=0.4]{square1g.eps}
%\caption{Outer billiard about the unit square.}
%\label{square_fig1}
%\end{center}
%\end{figure}
So $A_0=\{\Re(z)>0 \text{ and } \Im(z)>0\}$ is the positive quadrant of $ \Cc $. Let $T_{\la}$ be the outer billiard map on the square with $ \la \in (0,1) $. From now on, we will drop the label $\la$ from our notation. We have $T(A_0)\subset A_1\cup A_2$, and rotating clockwise $ A_1 \cap T(A_0) $ and $A_2 \cap T(A_0) $ about the center of the square by an angle of $90$ and $180$ degrees, respectively, we bring these regions back to $A_0$ (see Fig.~\ref{square_fig2}). 

We now compute explicitly the map $ f $ for the square. Let $ B=T^{-1}(A_{1} \cap T(A_{0}))$, and let $C=T^{-1}(A_{2} \cap T(A_{0}))$. These sets were denoted by $ D_{1} $ and $ D_{2} $ in Subsection~\ref{su:sk}. It is easy to check that $ B = \{\Re(z)>0, \, \Im(z)<1/\la\} $ and $ C = \{\Re(z)>0, \, \Im(z)>1/\la\} $ (see Fig.~\ref{square_fig2}).  
%Recall that $ f $ is the factor map of $ T $ (see Subsection~\ref{su:sk}). For the square, $R$ is the anti-clockwise rotation about the center of the square by an angle of $ 90 $ degree.} 

% Hence, $ f(B) =R(T(B)) = \{0<\Re(z)<1, \, \Im(z)>0\} $ and $ f(C) =R^{2}(T(C)) = \{\Re(z)>1, \, \Im(z)>0\} $. Note that $B$ and $C$ depend on $\la$, but $f(B)$ and $f(C)$ are independent of $ \lambda $ (see Fig.~\ref{square_fig3}).
%\begin{figure}[b]
%\begin{center}
%\includegraphics[scale=.6]{square3g.eps}
%%\begin{tikzpicture}[line cap=round,line join=round,>=triangle 45,x=1.0cm,y=1.0cm]
%%\clip(-0.85,0.64) rectangle (9.34,8.93);
%%\draw (7.68,2.03) node[anchor=north west] {$\operatorname{Re}(z)$};
%%\draw (0.3,8.53) node[anchor=north west] {$\operatorname{Im}(z)$};
%%\draw (-0.13,2.01) node[anchor=north west] {$0$};
%%\draw (-0.58,5.83) node[anchor=north west] {$\frac{i}{\lambda}$};
%%\draw (3.7,6.9) node[anchor=north west] {$f_{\lambda}(C)$};
%%\draw (1.88,1.98) node[anchor=north west] {$1$};
%%\draw [line width=1.2pt] (0,8)-- (0,2);
%%\draw [line width=1.2pt] (0,2)-- (8,2);
%%\draw (0.54,6.9) node[anchor=north west] {$f_{\lambda}(B)$};
%%\draw [line width=1.2pt,dash pattern=on 1pt off 1pt] (2,8)-- (2,2);
%%\draw [line width=1.2pt,dash pattern=on 1pt off 1pt] (0,5.1)-- (7.97,5.1);
%%\end{tikzpicture}
%\caption{Sets $ f_{\la}(B) $ and $ f_{\la}(C) $.}
%\label{square_fig3}
%\end{center}
%\end{figure}

Next, define $f|_B=f_B $ and $ f|_C=f_C $. Note that $ f_{A} $ and $ f_{B} $ are precisely the maps $ f_{1} $ and $ f_{2} $ defined in Subsection~\ref{su:sk}. A simple computation shows that 
\begin{equation}
\label{for_iteration_eq}
f_B(z)=i\la z+1, \qquad f_C(z)=\la z+1-i.
\end{equation}

Let $ C_{1} = B $, and let $ C_{n+1} = f^{-1}_{C}(C_{n}) \subset C $ for $ n \ge 1 $. For $ n \ge 2 $, the set $ C_{n} $ consists of elements of $ C $ hitting $ B $ exactly after $ n $ iterations of $ f $. In the next lemma, we prove that every non-singular point of $ C $ visits $ B $ eventually.

%Let $ \alpha_{1}(\lambda) = 0 $, and define
%\[
%\alpha_{n}(\lambda) = \lambda^{-1} + \cdots + \lambda^{-n+1} = \frac{1-\la^{n-1}}{\la^{n-1}(1-\la)} \qquad \text{for } n \ge 2.
%\]
%
%\begin{lemma}
%\label{le:cn}
%$ C_{n} = \left\{\Re(z) > 0, \, \alpha_{n}(\lambda) < \Im(z) < \alpha_{n+1}(\lambda) \right\} $.
%\end{lemma}
%
%\begin{proof}
%The proof is by induction. For $ C_{1} = B $, the claim is obvious. Now, suppose that it is true for $ C_{n} $. Let $ z \in C_{n} $. If $ z' = f^{-1}_{C}(z) \in C_{n+1} $, then $ z' \in C $, $ \Re(z') = \lambda^{-1}(\Re(z)-1) $ and $ \Im(z') = \lambda^{-1}(\Im(z)+1) $. From $ \Re(z) > 0 $, it follows that $ \Re(z')>-\lambda^{-1} $. But $ z' \in C $ so that $ \Re(z')>0 $. From $ \alpha_{n}(\lambda) < \Im(z) < \alpha_{n+1}(\lambda) $, it follows that $ \alpha_{n+1}(\lambda) < \Im(z') < \alpha_{n+2}(\lambda) $. This completes the proof.
%\end{proof}

\begin{lemma}
\label{le:b-global-section}	
$ C \setminus S_{\infty} \subset \bigcup_{n \ge 2} C_{n} $.
\end{lemma}

\begin{proof}
We argue by contradiction. Suppose that there exists $ z \in C $ such that $ z_{k}:=f^{k}_{C}(z) \in C $ for all $ k \ge 1 $. Thus $ \liminf_{k \to +\infty} \Im(z_{k}) \ge 1/\lambda $. However, by~\eqref{for_iteration_eq}, we have $ z_{k} \to (1-i)/(1-\lambda) $ as $k\to\infty$, yielding a contradiction.	
\end{proof}

Denote by $ B_{n} $ the subset of $ B $ formed by elements whose first return time to $ B $ is equal to $ n \ge 1 $. From the definition of $ C_{n} $, it follows that $ B_{n} = f^{-1}_{B}(C_{n}) \subset B $ for $ n \ge 1 $.
%\begin{lemma}
%\label{le:bn}
%$ B_{n} = \left\{\alpha_{n}(\lambda) < \lambda \Re(z) < \alpha_{n+1}(\lambda), \, 0 < \lambda \Im(z) < 1 \right\} $.
%\end{lemma}
%
%\begin{proof}
%Let $ z \in C_{n} $. If $ z' = f^{-1}_{B}(z) $, then $ z' \in B $ with $ \Re(z') = \lambda^{-1} \Im(z) $ and $ \Im(z') = \lambda^{-1}(1-\Re(z)) $. From $ \alpha_{n}(\lambda) < \Im(z) < \alpha_{n+1}(\lambda) $, it follows that $ \alpha_{n}(\lambda) < \la \Re(z') < \alpha_{n+1}(\lambda) $. From $ \Re(z)>0 $, it follows that $ \Im(z') < \lambda^{-1} $. But $ z' \in B $ so that $ 0 < \Im(z') < \lambda^{-1} $. 
%\end{proof}
Let $ B' = \bigcup_{n} B_{n} $, and denote by $ \phi \colon B' \to B $ the first return map to $ B $ induced by $ f $. Also, let $ \phi_{n} = \phi|_{B_{n}} $.

%\begin{lemma} 
%\label{le:first}
%$ \phi_{n}(z) = i\la^{n}z + \la^{n-1} + (1-i)(1 \cdots + \la^{n-1}) $ for $ z \in B_{n} $. 
%\end{lemma}
%
%\begin{proof}
%The claim can be proved by induction using~\eqref{for_iteration_eq}.
%\end{proof}

\begin{figure}[t]
%\begin{tikzpicture}[line cap=round,line join=round,>=triangle 45,x=1.0cm,y=1.0cm]
%\clip(-3.52,-0.19) rectangle (10.85,8.93);
%\draw (8.42,2.73) node[anchor=north west] {$\operatorname{Re}(z)$};
%\draw (-1.62,8.86) node[anchor=north west] {$\operatorname{Im}(z)$};
%\draw (-2.08,2.06) node[anchor=north west] {$0$};
%\draw (-2.82,5.97) node[anchor=north west] {$\frac{i}{\lambda}$};
%\draw (-0.28,2.13) node[anchor=north west] {$\frac{1}{\lambda^{2}}$};
%\draw [line width=1.2pt] (-2,2)-- (8,2);
%\draw (-1.23,4.1) node[anchor=north west] {$B_{1}$};
%\draw [line width=1.2pt] (-2,8)-- (-2,2);
%\draw [line width=1.2pt,dash pattern=on 1pt off 1pt] (-2,5.06)-- (7.99,5.03);
%\draw (2.47,2.1) node[anchor=north west] {$\frac{1}{\lambda^{2}}+\frac{1}{\lambda^{3}}$};
%\draw (5.57,2.13) node[anchor=north west] {$\frac{1}{\lambda^{2}}+\frac{1}{\lambda^{3}}+\frac{1}{\lambda^{4}}$};
%\draw (1.2,4.1) node[anchor=north west] {$B_{2}$};
%\draw (4.83,4.1) node[anchor=north west] {$B_{3}$};
%\draw [line width=1.2pt,dash pattern=on 1pt off 1pt] (0,2)-- (-0.02,5.03);
%\draw [line width=1.2pt,dash pattern=on 1pt off 1pt] (3,4.98)-- (3,2);
%\draw [line width=1.2pt,dash pattern=on 1pt off 1pt] (7.19,5)-- (7.2,2);
%\draw (2.82,7.94) node[anchor=north west] {$C$};
%\end{tikzpicture}
%\begin{center}
%\input{square4.pstex_t}
\includegraphics[scale=0.6]{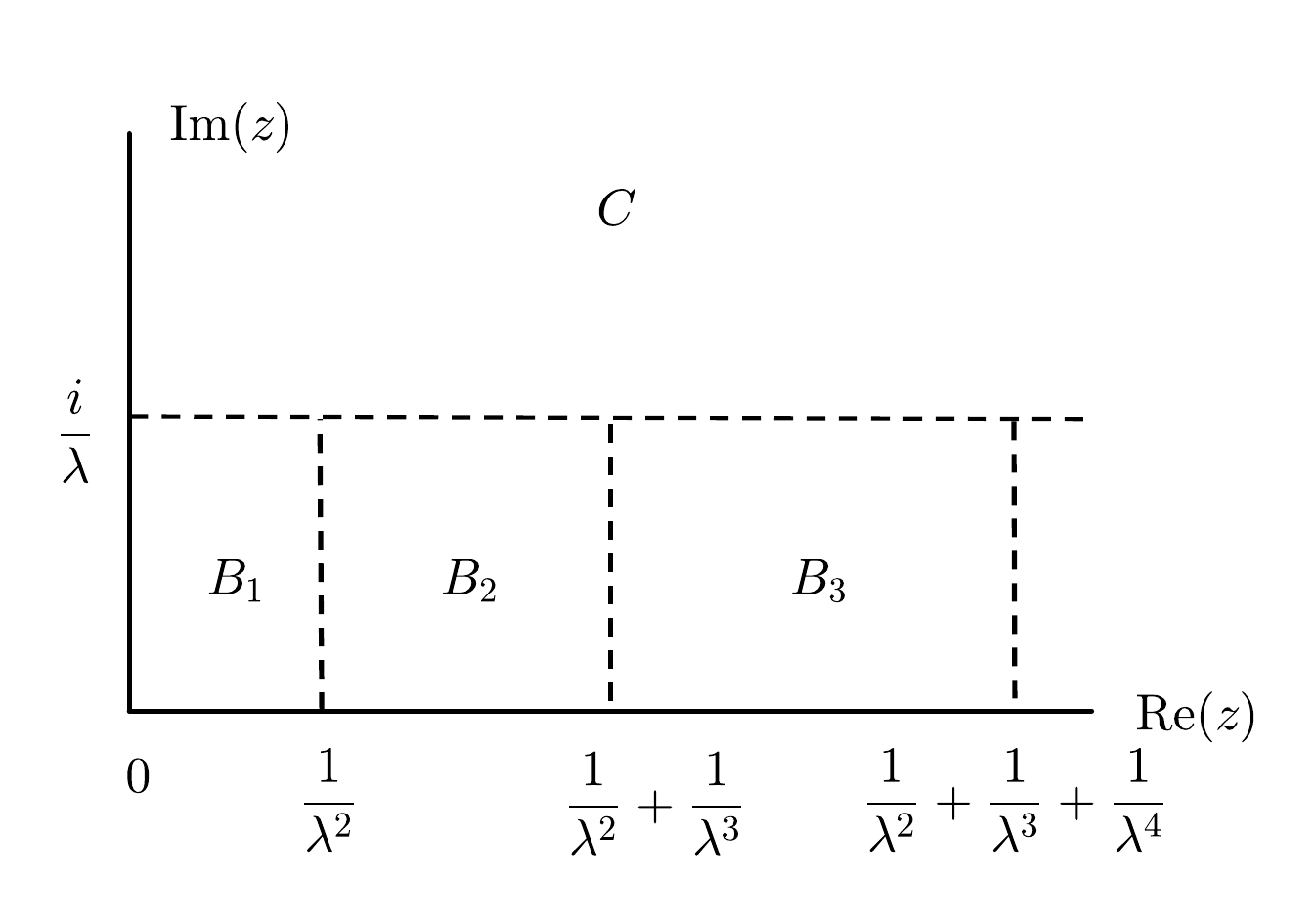}
\caption{Rectangles $ B_{n} $.}
\label{square_fig4}
\end{figure}

\begin{lemma}
\label{pr:monotonicity}	
We have
\begin{enumerate}
	\item $ \phi(B_{n}) = \phi_{n}(B_{n}) \subset \bigcup^{n}_{i=1} B_{i} $ for $ n \ge 1 $,
	\item $ \exists N = N(\lambda) $ such that $ \phi(B') \subset \bigcup^{N}_{i=1} B_{i} $.
\end{enumerate}
\end{lemma}

\begin{proof}
Similar to the proof of Lemma~\ref{lem:Bnmonotone}. See also Fig.~\ref{square_fig4}.
\end{proof}
%\begin{proof}
%Part~(1). Let $ z'=\phi_{n}(z) $ with $ z \in B_{n} $. By Lemma~\ref{le:bn}, it is enough to show that $ \Re(z') < \alpha_{n+1}(\lambda)/\lambda $. From Lemma~\ref{le:first}, it follows that $ \Re(z') = \lambda^{n}(\alpha_{n+1}(\lambda)-\Im(z)) $. Since $ z \in B $, we have $ 0<\Im(z)<\lambda^{-1} $, and so $ \Re(z') < \lambda^{n} \alpha_{n+1}(\lambda) < \alpha_{n+1}(\lambda)/\lambda $. 
%
%Part~(2). From the previous conclusion, we have $ \Re(z')< \lambda^{n}\alpha_{n+1}(\lambda) $. The claim now follows from $ \lambda^{n}\alpha_{n+1}(\lambda) = 1 + \cdots +\lambda^{n-1} \le (1-\lambda)^{-1} $.
%\end{proof}

Let $ \psi_{n} $ be the affine extension of $ \phi_{n} $ to the whole $ \Cc $. Since $ \psi_{n} $ is a strict contraction, it has a unique fixed point $ z_{n} $. Note that $ z_{n} $ is a fixed point of $ \phi_{n} $ (and so of $ \phi $) if and only if $ z_{n} \in B_{n} $. It is also clear that 
%By the definition of the maps $ f $ and $ \phi_{n} $, we see that 
if $ z_{n} $ is a fixed point of $ \phi_{n} $, then it is also a periodic point of $ f $ and $ T $.

The following proposition is an immediate consequence of Lemma~\ref{pr:monotonicity}.
%Clearly, the period of $ z_{n} $ as a periodic point of $ f $ is $ n $. The period of $ z_{n} $ as a periodic point of $ T $ will be determined in Corollary~\ref{co:final}.

\begin{proposition}
\label{pr:periodic}
Let $ N(\lambda) $ be the constant in Proposition~\ref{pr:monotonicity}. For every $ z \in B \setminus S_{\infty} $, there exists an integer $ 0< n \le N(\lambda) $ such that $ \lim_{k \to +\infty} \phi^{k}(z) = z_{n} \in \overline{B_{n}} $.
\end{proposition}

%\begin{proof}
%Let $ z \in B \setminus S_{\infty} $. From Part~(1) of Proposition~\ref{pr:monotonicity}, it follows that there exist two positive integers $ n=n(z) $ and $ j=j(z) $ such that $ \phi^{k}(z) \in B_{n} $ for $ k \ge j $. Hence, $ \phi^{k}(z) = \psi^{\, k-j}_{n}(\phi^{j}(z)) $, which implies $ \lim_{k \to +\infty} \phi^{k}(z) = z_{n} \in \overline{B_{n}} $. 
%\end{proof}

Using the terminology introduced in Section~\ref{intro}, we can rephrase the conclusion of the previous proposition by saying that $ \phi $ is asymptotically periodic. Let $ \omega(\phi) $ denote the $ \omega $-limit set of $ \phi $. The next corollary follows directly from Proposition~\ref{pr:periodic}. 

\begin{corollary}
\label{co:attractor}
We have $ \omega(\phi) \subset \{z_{1},\ldots,z_{N(\lambda)}\} $ for every $ \la\in(0,1) $.
\end{corollary}

A bifurcation analysis of the fixed points $ \{z_{n}\}_{n \ge 1} $ can be performed similarly to the corresponding analysis for the equilateral triangle. From this analysis, we can derive a precise description of $ \omega(\phi) $ and $ \omega(T) $. The following result can be proved by adapting the arguments used in Section~\ref{triangle}, in particular the proofs of Proposition~\ref{lem:fxpts} and Theorem~\ref{th:finaltrianglephi}. 

Let $\{\la_i\}_{i\geq1}$ be the sequence in Lemma~\ref{le:solutionQ}.
 
\begin{theorem}
\label{th:final}
Let $ m(\la)>0 $ be the integer defined by $ \la_{m(\la)} < \la \le \la_{m(\la)+1} $. Then $\omega(\phi)=\{z_1,\ldots,z_{m(\la)}\} $ for every $\la\in(0,1)$. 
\end{theorem}

%\begin{proof}
%For $ \la_{m} < \la < \la_{m+1} $, the wanted conclusion follows from Propositions~\ref{pr:periodic} and \ref{pr:fixed}. Thus, we need to show only that $ z_{m} \notin \omega(\phi) $ for $ \la = \la_{m} $. In this case, we have $ g_{m}(\la)=0 $, and from the computations in the proof of Proposition~\ref{pr:fixed}, it follows that $ \Re(z_{m}) = \alpha_{m}(\la)/\la $ and $ \Im(z_{m}) = 0 $, i.e., $ z_{m} $ lies on the lower left corner of the rectangle $ B_{m} $ (see Fig.~\ref{square_fig4}). Hence $ z_{m} \in \overline{B_{m}} \cap \overline{B_{m-1}} $. By Lemma~\ref{le:first}, the map $ \psi_{m} $ is the composition of an anti-clockwise rotation of $ 90 $ degree around $ z_{m} $ and a contraction toward $ z_{m} $,
%and so every point $ z \in B_{m} $ sufficiently close to $ z_{m} $ will enter $ B_{m-1} $ after at most two iterations of $ \psi_{m} $. Since $ z_{m-1} $ is strictly contained inside $ B_{m-1} $ for $ \la = \la_{m} $, the forward orbit of $ z $ will move away from $ z_{m} $. We conclude that $ z_{m} \notin \omega(\phi) $ for $ \la = \la_{m} $.
%\end{proof}
%
%Let $ m(\la) $ be the constant in Theorem~\ref{th:final}. 
%\deleted{Let $ \dist $ be the Euclidean distance on $ \Cc $.} 
The following is a corollary of Theorem~\ref{th:final} and Lemma~\ref{le:b-global-section}. Its proof is similar to the one of Corollary~\ref{co:finaltriangle}.

\begin{corollary}
\label{co:final}
For every $ \la \in (0,1) $, the map $ T $ has exactly $ m=m(\la) $ distinct periodic orbits $ \Gamma_{1},\ldots,\Gamma_{m} $, and the period of $ \Gamma_{i} $ is equal to $ 4i $. Moreover, for every non-singular $ z \in X $, there exists $ 1 \le i \le m $ such that $ \dist(T^{n}(z),\Gamma_{i}) \to 0 $ as $ n \to +\infty $. (see Fig.~\ref{fi:basin}).
\end{corollary}

%\begin{proof}
%Similar to the proof of Corollary~\ref{co:finaltriangle}.
%\end{proof}

\begin{remark}
It is worth observing that the bifurcation value $ \la_{2i-1} $ is the same for the periodic orbits of period $4(2i-1)$ of the dissipative outer billiard about the square and the periodic orbits of period $3(2i-1)$ of the dissipative outer billiard about the equilateral triangle.
\end{remark}

%
%\begin{proof}
%Let $ 0<\la<1 $. Since $ T $ and $ F $ (see Subsection~\ref{su:sk}) are conjugated, we prove the corollary with $ T $ replaced by $ F $.
%
%By Theorem~\ref{th:final} and Lemma~\ref{le:b-global-section}, it follows that $ O_{f}(z_{1}),\ldots,O_{f}(z_{m}) $ are the only periodic orbits of $ f $. Since $ z_{i} $ is a fixed point of $ \phi_{i} $, it follows from the definition of $ \phi_{i} $ that the period of $ O_{f}(z_{i}) $ is equal to $ i $, and that $ f^{i}(z_{i}) = f^{i-1}_{C} \circ f_{B}(z_{i}) $. An easy computation then shows that $ \sigma^{i}(z_{i}) = 2i-1 \mod 4 $. But $ 2i-1 $ and $ 4 $ are coprime, and so Lemma~\ref{le:periodic} implies that each set $ \Gamma_{i} := \pi^{-1}(O_{f}(z_{i})) $ is a periodic orbit of $ F $ of period $ 4 i $. Since $ f $ is a factor of $ F $, and all the periodic orbits of $ F $ are contained in $ \bigcup^{m}_{i=1} \pi^{-1}(O_{f}(z_{i})) $. 
%
%Now, let $ z $ be a non-singular point of $ F $. Theorem~\ref{th:final} implies that there exists $ 1 \le i \le m $ such that $ \dist(f^{n}(z),O_{f}(z_{i})) \to 0 $ as $ n \to \infty $. But $ f $ is a factor of $ F $, and so $ f^{n}(\pi(z)) = \pi(F^{n}(z)) $. It is trivial to check that $ \dist(v,\pi^{-1}(E)) = \dist(\pi(v),E) $ for every $ v \in X $ and every set $ E \subset A_{0} $. Combining together all the previous observations, we conclude that 
%\[
%\lim_{n \to +\infty} \dist(F^{n}(z),\Gamma_{i}) = \lim_{n \to +\infty} \dist(f_{n}(\pi(z)),O_{f}(z)_{i}) = 0.
%\]
%\end{proof}

\begin{figure}
\begin{center}
\includegraphics[trim=6cm 6cm 6cm 6cm, clip, width=6cm]{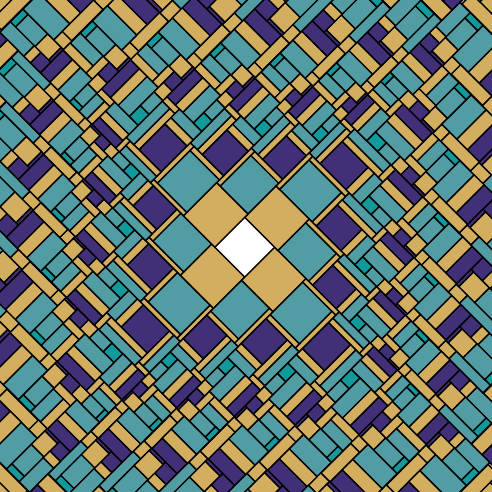}
\end{center}
\caption{Dissipative outer square billiard for $ \la = 0.95 $. The billiard map has three periodic orbits of period $ 4,8,12 $. The basin of attraction of each orbits is the union of all the regions of the same color.}
\label{fi:basin}
\end{figure}

\section{Persistency}
\label{persistency}

In this section, we prove that the dissipative square billiard remains asymptotically periodic under small perturbations of its vertices and the contraction rates. The precise statement is given in Theorem~\ref{th:persistent}.

Let $ \PP_{k} $ be the collection of all pairs $ (P,\vla) $ consisting of a convex $ k $-gon $ P $ with no more than two vertices lying on the same side and one vertex at the origin of $ \Cc $. Let $ \vla \in (0,1)^{k} $. We can naturally identify a $ k $-gon $ P $ with the vector $ \vw \in \R^{2k-2} $ whose components are the vertices $ w_{1},\ldots,w_{k-1} $ of $ P $ chosen so that if $ \partial P $ is positively oriented, then $ w_{1} $ follows the vertex at the origin, and $ w_{i+1} $ follows $ w_{i} $ for $ i=1,\ldots,k-1 $. In virtue of this identification, $ \PP_{k} $ is an open subset of $ \Rr^{3k-2} $ and a locally compact metric space with the sup norm $ \|\cdot\| $ of $ \Rr^{3k-2} $.

Now, let $ P $ be the square, and let $ \Lambda = \{\la_{1},\la_{2},\ldots\} $ be the set of all bifurcation values for the dissipative square billiards (see Theorem~\ref{th:final}). In this section, we will write $ T_{P,\vla} $ instead of $ T_{\vla} $ to emphasize the dependence of the billiard map on the polygon $ P $.

\begin{theorem}
\label{th:persistent}
Let $ P $ be the square, and suppose that $ \la_{i} = \la \in (0,1) \setminus \Lambda $ for $ i = 1,\ldots,4 $. Then for every $ \epsilon>0 $, there exists $ \delta>0 $ such that if $ (P',\vla') \in \PP_{4} $ and $ \|(P',\vla')-(P,\vla)\|<\delta $, then 
\begin{enumerate}
\item $ T_{P',\vla'} $ is asymptotically periodic,
\item $ \operatorname{card}(\omega( T_{P',\vla'})) = \operatorname{card}(\omega(T_{P,\vla})) $,
\item the Hausdorff distance of $ \omega(T_{P',\vla'}) $ and $ \omega(T_{P,\vla}) $ is less than $ \epsilon $,
\item corresponding periodic orbits of $ T_{P',\vla'} $ and $ T_{P,\vla} $ have the same period. 
\end{enumerate}
\end{theorem}

\begin{proof}
The theorem follows from Propositions~\ref{pr:cb-billiards} and \ref{pr:nointersection}.
\end{proof}

The next corollary is an immediate consequence of the previous theorem.

\begin{corollary}
The dissipative outer billiard about any quadrilateral sufficiently close to the square with contraction coefficients sufficiently close to some real in $ (0,1) \setminus \Lambda $ is asymptotically periodic. 

\end{corollary}

We now recall the notions of piecewise locally contracting maps and persistency introduced in~\cite{CB}. We present these notions only for two-dimensional maps, which is the situation of interest here.

Let $ K $ be a compact subset of $ \Cc $ with the metric $ |\cdot| $ inherited from $ \Cc $. We call\footnote{These transformations are called piecewise continuous locally contracting maps~\cite{CB}.} $ \Phi = \{\Phi_{i}\} $ a {\em piecewise contraction on $ K $} if there exist a constant $ 0 < \rho < 1 $ and a finite compact cover $ C=\{C_{1},\ldots,C_{m}\} $ of $ K $ such that $ C_{i} \cap C_{j} \subset \partial _{i} C_{i} \cap \partial C_{j} $ for $ i \neq j $, and $ \Phi_i \colon C_{i} \to \Phi(C_{i}) \subset K $ is a bijective strict contraction with Lipschitz constant $ L_{\Phi_{i}} \le \rho $ for every $ i $. Let $ L_{\Phi} = \max_{i} L_{\Phi_{i}} $. The map $ \Phi $ is defined by $ \Phi(A) = \bigcup_{i} \Phi_{i}(A \cap C_{i}) $ for every $ A \subset K $. The \textit{singular set} of $ \Phi $ is denoted by $\SSS_{\Phi}=\bigcup_{i\neq j} C_i\cap C_j$. Note that $ \Phi $ is multivalued transformation on $ \SSS_{\Phi} $. We will write $ (\Phi,C) $ to emphasize the role of the cover $ C $ in the definition of $ \Phi $.  

The space of all piecewise contractions on $ K $ is endowed with a topology defined by the following base. Given a piecewise contraction $ (\Phi,C) $ on $ K $ and a number $ \epsilon>0 $, the element $ B(\Phi,\epsilon) $ of the base is given by all piecewise contractions $ (\Psi,D) $ on $ K $ such that 
\begin{enumerate}
\item the covers $ C $ and $ D $ have the same cardinality,
\item $ \sup_{z \in C_{i} \cap D_{i}} |\Phi(z) - \Psi(z)| < \epsilon $ for every $ i $, 
\item $ |L_{\Phi}-L_{\Psi}|<\epsilon $,
\item the Haussdorf distance of $ C_{i} $ and $ D_{i} $ is less than $ \epsilon $ for every $ i $.
\end{enumerate}

Suppose that $ \Phi $ is piecewise contraction that is asymptotically periodic. We say that the $ \omega $-limit set of $ \Phi $ is \emph{persistent} provided that $ \Phi $ is asymptotically periodic, and that for every $ \epsilon>0 $, there exists $ \delta>0 $ such that if $ \Psi \in B(\Phi,\delta) $, then 
\begin{enumerate}
\item $ \Psi $ is asymptotically periodic,
\item $ \operatorname{card}(\omega(\Psi)) = \operatorname{card}(\omega(\Phi)) $,
\item the Hausdorff distance of $ \omega(\Psi) $ and $ \omega(\Phi) $ is less than $ \epsilon $,
\item corresponding periodic orbits of $ \Psi $ and $ \Phi $ have the same period. 
\end{enumerate}
 
The following proposition, proved by Catsigeras and Budelli, gives a sufficient condition for the persistency of $\om(\Phi)$~\cite[Lemma~3.3 and Remark~3.4]{CB}. 

\begin{proposition}
\label{pr:cb}
If $\Phi^n(K)\cap\SSS_{\Phi}=\emptyset$ for some integer $n>0$, then $ \Phi $ is asymptotically periodic, and $\om(\Phi)$ is persistent. 
\end{proposition}

%Remark: The set of maps contained in any given subset of the space of all piecewise contractions (including billiards of course) satisfying the condition Proposition~\ref{pr:cb} is always open (with respect to the relative topology).}

We now reformulate Proposition~\ref{pr:cb} in the context of billiards maps. Given $ (P,\vla) \in \PP_{k} $ for some $ k $, let $ K_{P,\vla} $ be the forward invariant compact set for the map $ T_{P,\vla} $ as in Proposition~\ref{basic_prop}, and let $ \SSS_{T_{P,\vla}} $ be the singular set of $ T_{P,\vla} $. 

\begin{proposition}
\label{pr:cb-billiards}
Let $ (P,\vla) \in \PP_{k} $, and suppose that $ T^{n}_{P,\vla}(K_{P,\vla}) \cap \SSS_{T_{P,\vla}} = \emptyset $ for some integer $n>0$. Then $ T_{P,\vla} $ is asymptotically periodic, and $ \om(T_{P,\vla}) $ is persistent with respect to the topology of $ \PP_{k} $. 
\end{proposition}

\begin{proof}
Let $ O $ be a compact neighborhood of $ \PP_{k} $ containing $ (P,\vla) $. Define
\[ 
a = \sup_{(P',\vla) \in O} \|\vla\| \qquad \text{and} \qquad b = \sup_{(P',\vla) \in O} \|\vw(P')\|,
\] 
where $ \vw(P') $ denotes the vector formed the vertices of $ P' $ as explained at the beginning of this section. Since $ O $ is compact, we have $ 0<a<1 $ and $ 0<b<\infty $. It follows easily from Proposition~\ref{basic_prop} that there exists a compact set $ W \subset \Cc $ forward invariant for every $ T_{P',\vla'} $ with $ (P',\vla') \in O $. Thus $ T_{P',\vla'} $ is a piecewise contraction\footnote{The map $ T_{P',\vla'} $ is not defined on $ P $. This can be easily fixed, by choosing $ u \in P $, and defining $ T_{P',\vla'}(z) = \min_{i} \la_{i} (z-u) $ for $ z \in P $. We also need to extend the map $ T_{P',\vla'} $ on each domain of continuity up to its boundary.} on $ W $ in the sense of Catsigeras and Budelli for every $ (P',\vla') \in O $. Now, denote by $ \Pi $ the transformation associating a pair $ (P',\vla') \in O $ to the piecewise contraction $ T_{P,\vla} $ on $ W $. If $ O $ is endowed with the Euclidean metric of $ \Rr^{3k-2} $, then it follows that $ \Pi $ is continuous. 

From the definition of $ W $, we see immediately that $ K_{P,\vla} $ is contained in $ W $. Then, Proposition~\ref{basic_prop} implies that $ T^{m}_{P,\vla}(W) \subset K_{P,\vla} $. Using $ T^{n}_{P,\vla}(K_{P,\vla}) \cap \SSS_{T_{P,\vla}} = \emptyset $, we obtain $ T^{n+m}_{P,\vla}(W) \cap \SSS_{T_{P,\vla}} = \emptyset $. The wanted conclusion now follows from Proposition~\ref{pr:cb} and the continuity of $ \Pi $. 
\end{proof}
  
Next, we show that the hypothesis of Proposition~\ref{pr:cb-billiards} is satisfied for the square with $ \la_{1} = \cdots = \la_{4} = \la \in (0,1) \setminus \Lambda $.
  
\begin{proposition}
\label{pr:nointersection}
Let $ P $ be the square, and suppose that $ \la_{i} = \la \in (0,1) \setminus \Lambda $ for $ i=1,\ldots,4 $. Then there exists an integer $ n > 0 $ such that $ T_{P,\vla}^{n}(K_{P,\vla}) \cap \SSS_{T_{P,\vla}} = \emptyset $. 
\end{proposition}

\begin{proof}
First, we extend $ T_{P,\vla} $ from $ A_{i} $ to $ \overline{A_{i}} $. This is clearly possible, because $ T_{P,\vla} $ is linear on each $ A_{i} $. Then, we redefine $ T_{P,\vla}(E) = \bigcup^{k}_{i=1} T_{P,\vla}(E \cap \overline{A}_{i}) $ for every set $ E \subset \overline{X} $. Note that $ T_{P,\vla} $ is now a multivalued map on $ \SSS_{T_{P,\vla}} $ with $ T_{P,\vla}(z) $ consists of at most of as many points as the number of singular lines meeting at $ z $.  The $ \om $-limit set $ \omega(z) $ has to be understood now as the set of accumulations points of the sequence of sets $ \{T^{n}_{P,\vla}(z)\}_{n \ge 0} $. Of course, we still have $ T_{P,\vla}(K_{P,\vla}) \subset K_{P,\vla} $.

For the extended map $ T_{P,\vla} $, Corollary~\ref{co:final} can be reformulated as follows: let $ m(\la) $ and $ \Gamma_{1},\ldots,\Gamma_{m(\la)} $ be the positive integer and the periodic orbits, respectively, as in Corollary~\ref{co:final}. If $ \la \notin \Lambda $, then $ \om(z) \subset \bigcup^{m(\la)}_{i=1} \Gamma_{i} $ for every $ z \in X $. This statement can be easily deduced from the proof of the corollary, which is similar to that of Corollary~\ref{co:finaltriangle}. Note that the statement holds for every $ z \in X $, and not just non-singular $ z \in X $ as in Corollary~\ref{co:final}. If $ \la \in \Lambda $, then the statement fails, because there exists a point $ z \in \SSS_{P,\vla} $ such that $ z \in T^{n}_{P,\vla}(z) $ for some $ n>0 $, and so $ \om(z) \cap \SSS_{P,\vla} \neq \emptyset $. The statement could have been easily modified to cover the case $ \la \in \Lambda $ as well, but that is not required for the purpose of this proof. 
%By including the point $ z $ in $ \om(T_{\vla}) $, the statement above would cover the general case $ \la \in (0,1) $. However, for technical reasons we prefer not to do so.  

Since each periodic orbit $ \Gamma_{i} $ is asymptotically stable, the previous reformulation of Corollary~\ref{co:final} implies that for every $ z \in K_{\vla} $, there exists a neighborhood $ U \subset X $ of $ z $ and an integer $ n > 0 $ such that $ T^{i}(U) \cap \SSS_{P,\vla} = \emptyset $ for all $ i \ge n $. By a standard compact argument, we conclude that $ T^{\bar{n}}_{P,\vla}(K_{P,\vla}) \cap \SSS_{P,\vla} = \emptyset $ for some integer $ \bar{n}>0 $. This completes the proof.
\end{proof}

\begin{remark}
Following Jeong~\cite{J}, one can show that the dissipative outer billiard about the square, the equilateral triangle and the hexagon is asymptotically periodic whenever $ \|\vla\|<1 $. The argument is quite simple, and we will illustrate it in the rest of this section. 

Let $T_\vla$ be the dissipative outer billiard about the unit square. When $\la_{1}=\cdots=\la_{k}=1$, the billiard has an explicit integral of motion $I(z)$ defined as follows. Given $z\in A_i$, let $x_i=x_i(z)$ and $y_i=y_i(z)$ be the unique positive numbers such that $z=w_i+x_i(w_{i-1}-w_i)+y_i(w_i-w_{i+1})$ for $ i=0,\ldots,3 $, where $ w_{4}=w_{0} $. Define $h(x,y)=[x] + [y]$, where $[a]$ is the integer part of $ a \in \Rr $. A simple computation shows that $I(T_\vla(z))=I(z)$ for $z\in X \setminus \SSS $, i.e., $I(z):=h(x_i(z),y_i(z))$ is an integral of motion. However when $\|\vla\|<1$, we have $ I(T_{\vla}(z)) \leq I(z)$, and $I$ becomes a Lyapunov function for $T_\vla$. Now, since the sequence of the vertices visited by an orbit contained in a set $ \{I=\alpha\} $ is completely determined by $ \alpha \in \Nn $, it is not difficult to show that the map $ T_{\vla} $ is asymptotically periodic for every $ \vla \in (0,1)^{4} $. 

The proof that the dissipative outer billiards about the equilateral triangle and the regular hexagon are both asymptotically periodic for $ \|\vla\|<1 $ is analogous. The Lyapunov function for both billiards is defined as for the square, but with $ h = h_{-} $ for the equilateral triangle, and with $ h = h_{+} $ for the hexagon, where
\[
h_{\pm}(x,y):=\left[\frac{x}{2}\right] + \left[\frac{y}{2}\right]+\left[\frac{\left|x\pm y\right|+1}{2}\right].
\]

To complete this remark, we observe that while the previous argument proves the asymptotic periodicity for the dissipative outer billiards considered, it does not give a complete description of the bifurcations of the $\om$-limit set, for which the analysis described in Sections~\ref{triangle} and \ref{square} or an alternative one is still required.
\end{remark}

\section{Concluding remarks}   
\label{conclu}

According to Bruin and Deane~\cite{Bruin}, almost every piecewise contraction is asymptotically periodic. This result, however, does not apply to polygonal dissipative outer billiards. The reason is that different polygons generate different partitions of the domain of definition of the corresponding outer billiard map, whereas in Bruin and Deane's setting, the partition is fixed a priori for the entire family of piecewise contractions. 

Not every dissipative outer polygonal billiard is asymptotically periodic. There exist quadrilaterals whose dissipative outer billiard has a Cantor set lying on the singular set, which attracts nearby points~\cite{J}. 

Numerical experiments we performed on dissipative outer billiards about regular polygons suggest the following conjecture.

\begin{conjecture}\label{conj:regular}
For almost every $\lambda\in(0,1)$, the dissipative outer billiard map $T_\lambda$ of any regular polygon is asymptotically periodic.\end{conjecture}

Let $K_{\vla}$ be the closed ball associated to $ T_{\vla} $ defined in Subsection~\ref{su:forward}. We say that the \textit{singular set of $T_\vla$ stabilizes} if there exists an integer $n\geq1$ such that $\mathcal{S}_{n+1} \cap K_{\vla} =\mathcal{S}_n \cap K_{\vla} $. 

\begin{proposition}\label{prop:stabilization}
If the singular set of $T_\vla$ stabilizes, then $\omega(T_\vla)$ is finite.
\end{proposition}

\begin{proof}
If the singular set of $T_\vla$ stabilizes then there exists an integer $n\geq1$ such that the map $T^k_\vla$ is continuous on each connected component of $K_\vla\setminus \mathcal{S}_n$ for every $k > 0$. Since $K_\vla\setminus \mathcal{S}_n$ has finitely many such components, say $Y_1,\ldots,Y_m$, and the map $T_\vla$ contracts uniformly every $Y_i$ into some $Y_j$, we conclude that $\omega(T_\vla)$ is finite.
\end{proof}

We cannot claim that every element of $ \om(T_{\vla}) $ in the previous proposition is a periodic orbit, because some of these elements may belong to the singular set of $ T_{\vla} $. However, such an element does generate a periodic orbit for a proper extension of $ T_{\vla} $ up to $ \SSS $ (degenerate periodic orbit). We believe that degenerate periodic orbits exist only in very specific circumstances, for example at the bifurcation of a true periodic orbit. Even though the stabilization of the singular set is not strictly speaking a sufficient condition for the asymptotic stability, it still provides a useful criterion for investigating asymptotic periodicity that can be effectively implemented numerically. 

We wrote a script in \textit{Mathematica} that generates the singular sets of any order of a polygonal outer billiard, and checks whether they stabilize. When the singular set stabilizes at some order $n$, the script attributes a color to the connected components of $K_\vla\setminus \mathcal{S}_n$. The union of all regions of the same color gives the basin of attraction in $K_\vla$ of a single attracting periodic orbit. 

To test numerically Conjecture~\ref{conj:regular}, we run our code with different regular $k$-gons with $5\leq k \leq 12$, and different contraction rates $ \lambda_{i} \in (0,1) $. As the number of sides of the polygon increases or the contraction rates get closer to $1$, the code execution time increases drastically. All experiments showed that the singular set stabilizes, in agreement with our conjecture. 

Figs~\ref{fig:pentagon}-\ref{fig:octnon} show some snapshots of our numerical experiments. In these figures, the polygon $P$ is located at the center in white color. The contraction rates $ \la_{i} $ are all equal to $ \la \in (0,1) $. The set $ K_{\vla} $ surrounds $ P $, and the lines forming the stabilized singular set $ \SSS_{n} $ are drawn in black color. The domains of continuity of $T^n_\vla$ are the polygonal regions cut by $ \SSS_{n} $. The different colors indicate the basins of attraction of different periodic orbits.

The basins of attractions of the periodic orbits in the dissipative outer billiard about the pentagon can be seen in Fig.~\ref{fig:pentagon}: the four colors correspond to the basins of attractions of four distinct attracting periodic orbits. 
\begin{figure}[t]
\subfloat[]{\includegraphics[trim=80mm 80mm 80mm 80mm, clip, width=5.5cm]{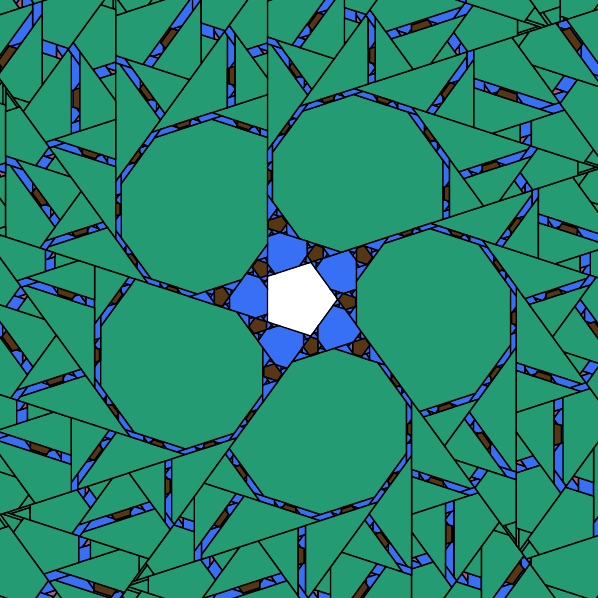}}
\hspace{1cm}
\subfloat[]{\includegraphics[trim=10mm 10mm 10mm 10mm, clip, width=5.5cm]{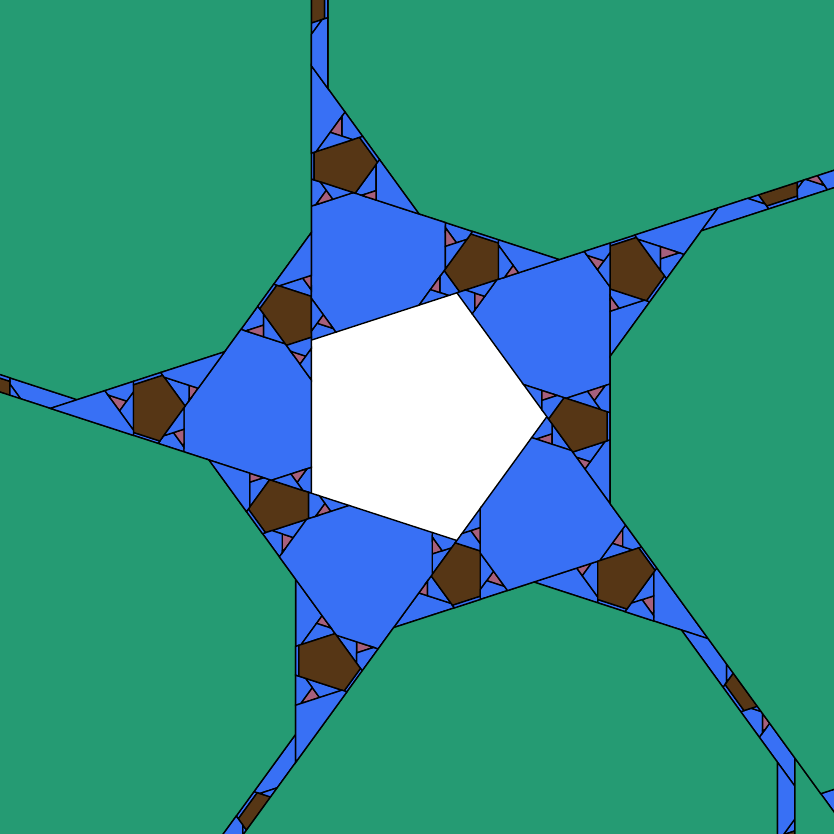}}
\caption{(A) Pentagon with $\la=0.95$ and (B) magnification of (A).}
\label{fig:pentagon}
\end{figure}
In Fig.~\ref{fig:pentagon}(B), the blue regions form the basin of attraction of the Fagnano orbit. In the figure, there is also another periodic orbit of period five, but winding around the pentagon twice. This orbit is located on the largest polygonal regions in green color. The basin of attraction of a period orbit of period $10$ is plotted in brown color. Finally, the tiny regions in pink color form the basin of attraction of a period orbit of period $35$. 

We found similar orbit structures for the dissipative outer billiard about the hexagon, heptagon, octagon and nonagon (see Figs~\ref{fig:hexhep} and \ref{fig:octnon}).

\begin{figure}[ht]
\subfloat[]{\includegraphics[trim=60mm 60mm 60mm 60mm, clip, width=5.5cm]{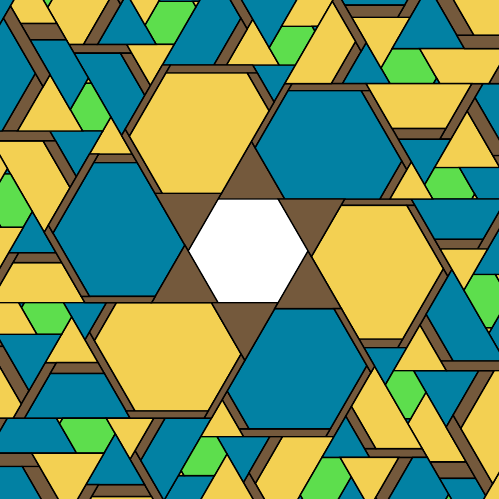}}
\hspace{1cm}
\subfloat[]{\includegraphics[trim=65mm 65mm 65mm 65mm, clip, width=5.5cm]{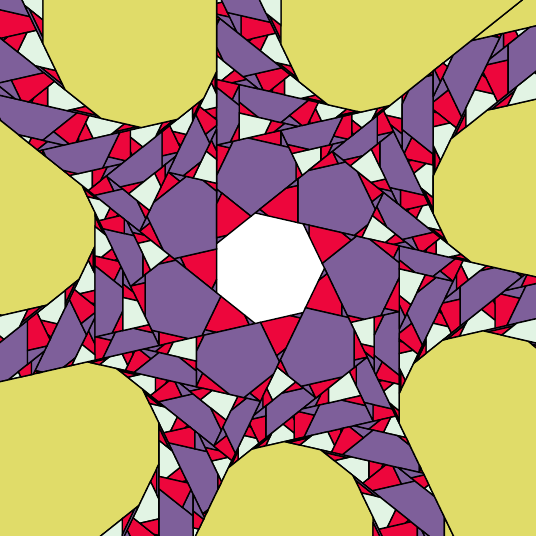}}
\caption{(A) Hexagon and (B) Heptagon with $\la=0.90$.}
\label{fig:hexhep}
\end{figure}

\begin{figure}[ht]
\subfloat[]{\includegraphics[trim=100mm 100mm 100mm 100mm, clip, width=5.5cm]{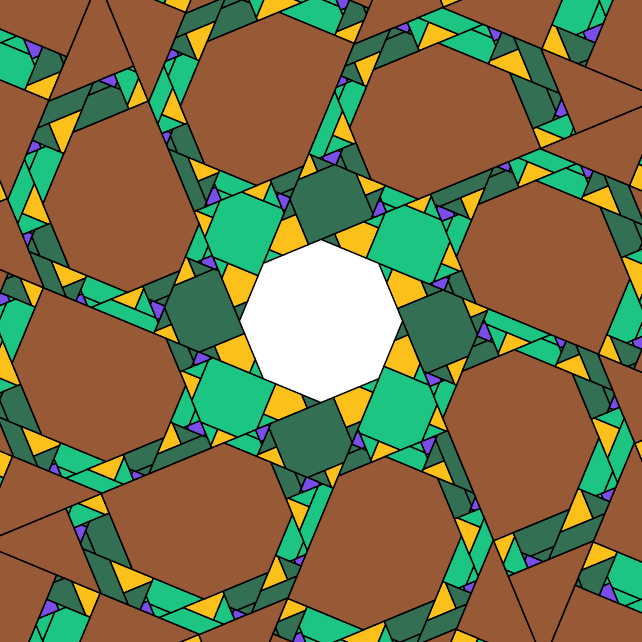}}
\hspace{1cm}
\subfloat[]{\includegraphics[trim=60mm 60mm 60mm 60mm, clip, width=5.5cm]{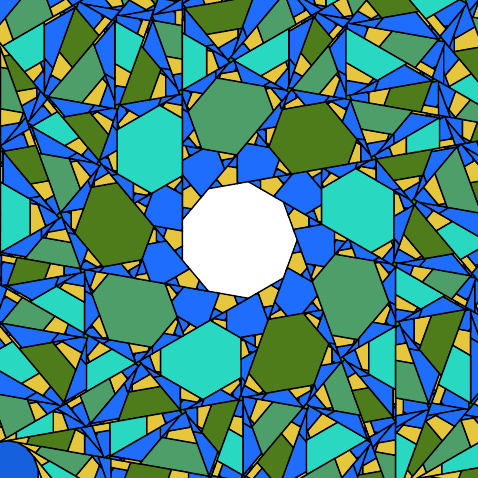}}
\caption{(A) Octagon and (B) Nonagon with $\la=0.90$.}
\label{fig:octnon}
\end{figure}

\newpage

\appendix
\section{}
%Definition of $\la_n$ and $\gamma_n$}
\label{appendix}

\begin{lemma}
\label{le:solutionQ}
The polynomial $$ p_{n}(x) = x^{2n} - x^{n} - x^{n-1} + 1\,,\quad n\geq1 $$ has a unique root $ \la_{n} \in [0,1) $. Moreover,
\begin{enumerate}
	\item[(a)] $p_n(x)<0$ for every $x\in(\la_n,1)$
	\item[(b)] $\{\la_{n}\}$ is strictly increasing
	\item[(c)] $\lim_{n\to\infty} \la_{n}=1$
\end{enumerate}
\end{lemma}

\begin{proof}
Since $p_n(1)=0$ we can write $p_n(x)=(x-1)q_n(x)$ where
$$
q_n(x)=x^{2n-1}+\cdots+x^{n}-x^{n-2}-\cdots-1\,.
$$
When $n=1$, we have $q_1(x)=x$, and so $\la_1=0$. Now, let $ n>1 $. By Descartes' rule of signs, the polynomial $q_n$ has a unique positive root, say $\la_n> 0$. Since $q_n(0)=-1$ and $q_n(1)=1$ we conclude that $\la_n\in(0,1)$. Part (a) follows from the fact that $q_n(x)>0$ for every $x\in(\la_n,1)$. To prove Part (b), we show that $ q_{n}(\la_{n+1})>0 $. Indeed,
\begin{align*}
q_{n}(\la_{n+1})&=\la_{n+1}^{2n-1}+\cdots+\la_{n+1}^{n}-\la_{n+1}^{n-2}-\cdots-1\\
&=q_{n+1}(\la_{n+1})+\la_{n+1}^{n-1}+\la_{n+1}^{n}-\la_{n+1}^{2n}-\la_{n+1}^{2n+1}\\
&=\la_{n+1}^{n-1}(1+\la_{n+1}-\la_{n+1}^{n+1}-\la_{n+1}^{n+2})>0\,.
\end{align*}
Finally, we prove Part (C). Let $ \la_\infty $ be the limit of $ \{\la_{n}\} $. Such a limit exists, because the sequence is strictly increasing and bounded. If $ \la_\infty<1 $, then $\lim_{n\to\infty}q_n(\la_\infty)=-1$, contradicting the fact that $q_n(\la_\infty)>0$ for every $n$. This completes the proof.
%Suppose that $ g_{n}(\la)=0 $ has more than one solution in $ [0,1) $. Since $ g_{n}(1) = 0 $, this would imply that the derivative $ g'_{n} $ vanishes at more than one point of $ (0,1) $. Note that $ g'_{n}(\la) = \la^{n-2} h_{n}(\la) $, where $ h_{n}(\la) = 2n \la^{n+1} - n \la - n + 1 $. Since $ h_{n}(0) = -n+1 \le 0 $, $ h_{n}(1) = 1 $ and $ h''_{n}(\la)>0 $ for all $ \la \in (0,1) $, it follows that $ h_{n} $ vanishes exactly at one point of $ (0,1) $. The same is then for $ g'_{n} $, and we obtain a contradiction.
%
%To show that $ \la_{n} < \la_{n+1} $, we claim that $ g_{n}(\la_{n+1})<0 $. Indeed,
%\begin{align*}
%0 & = g_{n+1}(\la_{n+1}) = \la^{2n+2}_{n+1} - \la^{n+2}_{n+1} - \la^{n+1}_{n+1} +1 \\
%& = \la_{n+1} (\la_{n+1} \la^{2n}_{n+1} - \la^{n}_{n+1} + \la^{n-1}_{n+1}) + 1 \\
%& > \la_{n+1} \la^{2n}_{n+1} - \la^{n}_{n+1} + \la^{n-1}_{n+1} + 1 \\
%& > \la^{2n}_{n+1} - \la^{n}_{n+1} + \la^{n-1}_{n+1} + 1 \\
%& = g_{n}(\la_{n+1}).
%\end{align*}
%
%Finally, let $ \bar{\la} $ be the limit of $ \{\la_{n}\} $. Such a limit exists, because the sequence is strictly increasing. If $ \bar{\la}<1 $, then it is easy to see that $ g_{n}(\la_{n}) \to 1 $, contradicting the fact that $ g_{n}(\la_{n}) = 0 $ for every $ n $.
\end{proof}

\begin{lemma}
\label{le:solutionT}
The polynomial $$ p_{n}(x) = x^{4n+4}-x^{3n+3}+x^{2n+2}-x^{2n+1}-x^{n+1}+1\,,\quad n\geq0 $$ has a unique root $ \gamma_{n} \in [0,1) $. Moreover, 
\begin{enumerate}
	\item[(a)] $p_n(x)<0$ for every $x\in(\gamma_n,1)$
	\item[(b)] $\{\gamma_{n}\}$ is strictly increasing
	\item[(c)] $\lim_{n\to\infty} \gamma_{n}=1$
\end{enumerate}
\end{lemma}

\begin{proof}
Since $p_n(1)=0$ we can write $p_n(x)=(x-1)q_n(x)$ where
$$
q_n(x)=x^{4n+3}+\cdots+x^{3n+3}+x^{2n+1}-x^{n}-\cdots-1\,.
$$
By Descartes' rule of signs, the polynomial $q_n$ has a unique positive root, say $\gamma_n> 0$. Since $q_n(0)=-1$ and $q_n(1)=1$ we conclude that $\gamma_n\in(0,1)$. Part (a) follows from the fact that $q_n(x)>0$ for every $x\in(\gamma_n,1)$. To prove Part (b), we show that $ q_{n}(\gamma_{n+1})>0 $. Indeed,
\begin{align*}
q_{n}(\gamma_{n+1})&=\gamma_{n+1}^{4n+3}+\cdots+\gamma_{n+1}^{3n+3}+\gamma_{n+1}^{2n+1}-\gamma_{n+1}^{n}-\cdots-1\\
&=q_{n+1}(\gamma_{n+1})+r_n(\gamma_{n+1})
\end{align*}
where 
\begin{align*}
r_n(x)&=x^{n+1}+x^{2n+1}+x^{3n+3}+\cdots+x^{3n+5}\\
&-x^{2n+3}-x^{4n+4}-x^{4n+5}-\cdots-x^{4n+7}\,.
\end{align*}
Since $r_n(\gamma_{n+1})>0$ and $q_{n+1}(\gamma_{n+1})=0$ we get $q_{n}(\gamma_{n+1})>0$. To complete the proof, we prove Part (c). Let $ \gamma_\infty $ be the limit of $ \{\gamma_{n}\} $. Such a limit exists, because the sequence is strictly increasing and bounded. If $ \gamma_\infty<1 $, then $\lim_{n\to\infty}q_n(\gamma_\infty)=-1$, contradicting $q_n(\gamma_\infty)>0$ for every $n$.
\end{proof}
%%%%%%%%%%%%%%%%%%%%%%%%%%%%%%%%%%%%%%%%%%%%%%%%%%%%%%%%%%%%%%%%%%%%%%%%%%%%

%\bibliographystyle{plain}
%\addcontentsline{toc}{section}{References}
%\bibliography{rfrncs}

\bibliographystyle{plain}

\end{document}